\documentclass[11pt]{article}
\usepackage[tbtags]{amsmath}
\usepackage{amssymb}
\usepackage{amsthm}
\usepackage[misc]{ifsym}
\usepackage{cases}
\usepackage{mathrsfs}
\usepackage{graphicx}
\usepackage{float}

\numberwithin{equation}{section}
\normalsize

\title{\bf Linear-Quadratic Mean Field Games \\with Common Noise: A Direct Approach \thanks{This work is supported by National Key R$\&$D Program of China (2022YFA1006104), National Natural Science Foundations of China (12471419, 12271304, 62192753), and Shandong Provincial Natural Science Foundations (ZR2024ZD35, ZR2022JQ01, ZR2022JQ31).}}

\author{\normalsize  Wenyu Cong\thanks{\it School of Mathematics, Shandong University, Jinan 250100, P.R. China, E-mail: congwenyu@mail.sdu.edu.cn} , Jingtao Shi\thanks{\it Corresponding author. School of Mathematics, Shandong University, Jinan 250100, P.R. China, E-mail: shijingtao@sdu.edu.cn}, Bingchang Wang\thanks{\it School of Control Science and Engineering, Shandong University, Jinan 250061, P.R. China, E-mail: bcwang@sdu.edu.cn}}

\newtheorem{mypro}{Proposition}[section]
\newtheorem{mythm}{Theorem}[section]
\newtheorem{mydef}{Definition}[section]

\newtheorem{Remark}{Remark}[section]
\newtheorem{example}{Example}[section]

\begin{document}

    \maketitle

    \noindent{\bf Abstract:}\quad
    This paper investigates a linear-quadratic mean field games problem with common noise, where the drift term and diffusion term of individual state equations are coupled with both the state, control, and mean field terms of the state, and we adopt the direct approach to tackle this problem.
    Compared with addressing the corresponding mean field teams problem, the mean field games problem with state coupling presents greater challenges.
    This is not only reflected in the explosive increase in the number of adjoint equations when applying variational analysis but also in the need for more Riccati equations during decoupling the high-dimensional forward-backward stochastic differential equations system.
    We take a different set of steps and ingeniously utilize the inherent properties of the equations to address this challenge.
    First, we solve an $N$-player games problem within a vast and finite population setting, and obtain a set of forward-backward stochastic differential equations by variational analysis.
    Then, we derive the limiting forward-backward stochastic differential equations by taking the limit as $N$ approaches infinity and applying the law of large numbers.
    Based on the existence and uniqueness of solutions to backward stochastic differential equations, some variables in the equations are identically zero, which significantly reduces the complexity of the analysis.
    This allows us to introduce just two Riccati equations to explicitly construct decentralized strategies for all participants.
    Moreover, we demonstrate that the constructed decentralized strategies constitute an $\epsilon$-Nash equilibrium strategy for the original problem.
    We also extend the results to the infinite-horizon case and analyze the solvability of algebraic Riccati equations.
    Finally, numerical simulations are provided to illustrate the preceding conclusions.

    \vspace{2mm}


    \vspace{2mm}


    \section{Introduction}

    {\it Mean field games} (MFGs) have garnered increasing academic interest and have been applied in numerous areas, including economics, smart grids, engineering, and social sciences (\cite{Bensoussan-Frehse-Yam-13}, \cite{Gomes-Saude-14}, \cite{Caines-Huang-Malhame-17}, \cite{Carmona-Delarue-18}).
    MFGs theory offers a framework for understanding the behavior of large-population models.
    In these models, while individual entities have negligible influence, the overall population's impact is significant.
    The theory explores the existence of Nash equilibria by studying the connection between finite-population and infinite-population limit problems.

    The methodological principles of MFGs were first introduced by Lasry and Lions \cite{Lasry-Lions-07} and independently by Huang et al. \cite{Huang-Malhame-Caines-06}.
    They have proven to be effective and manageable for analyzing weakly coupled stochastic controlled systems with mean field interactions, and for establishing approximate Nash equilibria.
    In particular, within the {\it linear-quadratic} (LQ) framework, MFGs provide a flexible modeling tool that can be applied to a wide range of practical problems.
    LQ-MFGs offer solutions with notable and elegant properties.
    Current academic research has extensively explored MFGs, especially within the LQ framework (\cite{Li-Zhang-08}, \cite{Bensoussan-Sung-Yam-Yung-16}, \cite{Huang-Wang-Wu-16}, \cite{Moon-Basar-17}, \cite{Huang-Zhou-20}, \cite{Xu-Zhang-20}, \cite{Wang-Zhang-Zhang-20}, \cite{Bensoussan-Feng-Huang-21}).
    In \cite{Cong-Shi-25}, the authors investigate indefinite LQ-MFGs in stochastic large-population systems. They employ the direct approach to address this problem and demonstrate that the resulting decentralized strategy serves as an $\varepsilon$-Nash equilibrium.
    \cite{Cong-Shi-24} explores LQ Stackelberg MFGs for backward-forward stochastic systems, which include a backward leader and numerous forward followers.
    \cite{Cong-Shi-Wang-24} examines LQ Stackelberg MFGs and {\it mean field teams} (MFTs) with any population sizes, where the followers' game is divided into non-cooperative and cooperative types, with follower number being either finite or infinite.

    Two primary approaches are traditionally used to solve MFGs.
    One is the fixed-point approach (or top-down approach, NCE approach, see \cite{Huang-Malhame-Caines-06}, \cite{Huang-Caines-Malhame-07}, \cite{Li-Zhang-08}, \cite{Bensoussan-Frehse-Yam-13}, \cite{Carmona-Delarue-18}, \cite{Wang-Huang-19}), which begins by using mean field approximation to formulate a fixed-point equation.
    Decentralized strategies can be developed by solving this equation and analyzing a representative player's optimal response.
    The other is the direct approach (or bottom-up approach, refer to \cite{Lasry-Lions-07}, \cite{Wang-Zhang-Zhang-20}, \cite{Huang-Zhou-20}, \cite{Wang-Zhang-Zhang-22}, \cite{Wang-24}, \cite{Cong-Shi-24}, \cite{Cong-Shi-25}, \cite{Wang-25}, \cite{Wang-Xu-Zhang-Liang-25}, \cite{Liang-Wang-Zhang-25}), which starts by solving an $N$-player games problem in a large-population setting.
    Then, centralized control depending on a specific player's state and the population's average state can be derived by decoupling or reducing high-dimensional systems.
    As the population size $N$ tends to infinity, constructing decentralized strategies becomes feasible.
    In \cite{Huang-Zhou-20}, the authors explored the connection and difference between these two approaches in an LQ setting.
    In passing, the de-aggregation approach (see \cite{Wang-Zhang-Zhang-22}, \cite{Wang-Zhang-Fu-Liang-23}, \cite{Liang-Wang-Zhang-24}, \cite{Cong-Shi-Wang-24}), evolving from the direct approach, takes the conditional expectation with respect to the information filter adapted to the $i$th agent's decentralized control set and de-aggregates the mean field term to obtain a decentralized strategy.
    These strategies are exact Nash equilibria for games with any population sizes, differing from the asymptotic optimality of the fixed-point and direct approaches for large populations.

    In this paper, we explore LQ-MFGs with common noise, where individual state equations are coupled with the population's average state.
    Initially, we resolve an $N$-player games problem in a large-population setting via variational analysis, deriving a set of {\it forward-backward stochastic differential equations} (FBSDEs).
    Then, we can sidestep the complex decoupling procedure and only need to introduce two Riccati equations by taking the limit as $N \to \infty$ and employing the law of large numbers.
    This allows us to explicitly construct decentralized strategies for all participants in both finite-horizon and infinite-horizon scenarios, overcoming the complexity of high-dimensional Hamiltonian systems caused by the coupling in the state equation.
    We demonstrate that the constructed decentralized strategies constitute an $\epsilon$-Nash equilibrium for the original problem.

    The key contributions of this paper are summarized as follows.

    $\bullet$ Motivated by prior work \cite{Huang-Zhou-20}, \cite{Liang-Wang-Zhang-25}, \cite{Wang-25}, we further consider a model with common noise in the state equation, coupled with the population's average state.
    The overall model is more general.
    We also analyze the corresponding infinite-horizon problem.

    $\bullet$ Diverging from conventional fixed-point methods like \cite{Xu-Zhang-20}, we adopt a direct approach from \cite{Huang-Zhou-20}, \cite{Wang-Zhang-Zhang-20} to reduce the complexitiy of our games problem.
    This involves decoupling the high-dimensional Hamiltonian system via mean field approximations to formulate decentralized strategies for all players.
    These strategies are later proven to form an $\epsilon$-Nash equilibrium.

    $\bullet$ Compared with the MFTs problem in \cite{Wang-Zhang-Zhang-22}, MFGs problem exhibit greater complexity in terms of coupling relationships and present more significant challenges in terms of processing.
    In MTFs problem, all agents collaborate to optimize a common performance index, resulting in a relatively straightforward optimization problem.
    However, in MFGs problem, agents act independently and make decisions based on their own interests.
    This leads to a complex interaction between individual optimal strategies and collective optimal strategies.

    $\bullet$
    Unlike the cumbersome procedure in \cite{Si-Shi-25}, we derive the limiting FBSDEs by taking the limit as $N$ approaches infinity and applying the law of large numbers.
    Notably, some variables in these equations are identically zero, which drastically reduces their dimensionality.
    This allows us to introduce just two Riccati equations instead of four, like dealing with MFTs problem, to explicitly construct decentralized strategies for all participants, thereby effectively overcoming the complexity of high-dimensional Hamiltonian systems that arises from the coupling in the state equation.

    The paper is structured as follows.
    Section 2 establishes the problem formulation.
    In Section 3, we first investigate the finite-horizon problem and then extend the results to the infinite-horizon case.
    In Section 4, numerical examples are given to validate the proposed strategies.
    Finally, conclusions are presented in Section 5.

    The following notations will be used throughout this paper.
    We use $||\cdot||$ to denote the norm of a Euclidean space, or the Frobenius norm for matrices.
    For a symmetric matrix $Q$ and a vector $z$, $||z||^2_Q \equiv z^\top Qz$.
    We use $c$ to denote a positive constant independent of the population size, and maybe different from line to line.
    For any real-valued scalar functions $f(\cdot)$ and $g(\cdot)$ defined on $\mathbb{R}$, $f(x) = O(g(x))$ means that there exists a constant $c > 0$ such that $\lim_{x \to \infty} |\frac{f(x)}{g(x)}| = c$, where $|\cdot|$ is an absolute value, which is also equivalent to saying that there exist $c > 0$ and $x$ such that $|f(x)| \le c|g(x)|$ for any $x \ge x'$.
    $\mathcal{R}(\cdot)$ denotes the range of a matrix or a vector.
    Let $T>0$ be a finite time duration and $(\Omega, \mathcal{F}, \{\mathcal{H}_t\}_{0 \le t \le T}, \mathbb{P})$ be a complete filtered probability space with the filtration $\{\mathcal{H}_t\}_{0 \le t \le T}$ augmented by all the $\mathbb{P}$-null sets in $\mathcal{F}$. $\mathbb{E}$ denoted the expectation with respect to $\mathbb{P}$. Let $L_{\mathcal{H}}^2(0,T;\cdot)$ be the set of all vector-valued (or matrix-valued) $\mathcal{H}_t$-adapted processes $f(\cdot)$ such that $\mathbb{E}\int^T_0 ||f(t)||^2dt < \infty$ and $L^2_{\mathcal{H}_t}(\Omega;\cdot)$ be the set of $\mathcal{H}_t$-measurable random variables, for $t\in[0,T]$.

    \section{Problem formulation}

    We consider a large-population system with $N$ homogeneous agents.
    Since the dynamics and cost functionals of these agents remain unchanged regardless of the indexing applied, the system exhibits symmetry, and the players are exchangeable.
    Here, $N$ can be arbitrarily large.

    For the sake of clarity and coherence of this paper, we will first present the problem setup for the finite-horizon. The infinite-horizon problem will be detailed in the corresponding section.

    The states equation of the $i$th agent, is given by the following SDE:
    \begin{equation}\label{state}
        \left\{
        \begin{aligned}
            dx_i(t) =& \big[A(t)x_i(t) + B(t)u_i(t) + G(t)x^{(N)}(t) + f(t)\big]dt\\
                     & + \big[C(t)x_i(t) + D(t)u_i(t) + F(t)x^{(N)}(t) + \sigma(t)\big]dW_i(t)\\
                     & + \big[C_0(t)x_i(t) + D_0(t)u_i(t) + F_0(t)x^{(N)}(t) + \sigma_0(t)\big]dW_0(t),\quad t\in[0,T],\\
            x_i(0) =&\ \xi_i ,
        \end{aligned}
        \right.
    \end{equation}
    where $1 \le i \le N$, $x_i \in \mathbb{R}^n$, $u_i \in \mathbb{R}^r$ and $\xi_i$ are the state process, control process and initial value of the $i$th agent, $x^{(N)}(t) := \frac{1}{N} \sum_{j=1}^{N} x_j(t)$ is called the state average or mean field term of all agents.
    $W_i(t), i = 0,\cdots,N$ are a sequence of independent one-dimensional Brownian motions defined on $(\Omega, \mathcal{F}, \{\mathcal{F}_t\}_{0 \le t \le T}, \mathbb{P})$.
    Here, $A(\cdot)$, $B(\cdot)$, $G(\cdot)$, $C(\cdot)$, $D(\cdot)$, $F(\cdot)$, $C_0(\cdot)$, $D_0(\cdot)$ and $F_0(\cdot)$ are deterministic matrix-valued functions with compatible dimensions.
    $f(\cdot)$, $\sigma(\cdot)$ and $\sigma_0(\cdot)$ are $\mathcal{F}_t^0$-adapted $\mathbb{R}^n$ vector-valued processes, reflecting the impact of the environment on each agent, where $\mathcal{F}_t^0$ is the $\sigma$-algebra generated by the common noise $W_0(\cdot)$.

    The cost functional of the $i$th agent is given by
    \begin{equation}\label{finite cost}
    \begin{aligned}
        J^N_i(u_i(\cdot), u_{-i}(\cdot)) &= \mathbb{E} \bigg\{\int_0^T \left[\big|\big|x_i(t) - \varGamma(t) x^{(N)}(t) - \eta(t)\big|\big|^2_{Q(t)} + \big|\big|u_i(t)\big|\big|^2_{R(t)} \right]dt \\
            &\qquad\qquad + \big|\big|x_i(T) - \varGamma_0 x^{(N)}(T) - \eta_0\big|\big|^2_H\bigg\},
    \end{aligned}
    \end{equation}
    where $Q(\cdot)$, $R(\cdot)$ and $\varGamma(\cdot)$ are deterministic matrix-valued functions with compatible dimensions, $H$ and $\varGamma_0$ are deterministic matrices, $\eta(\cdot)$ is an $\mathcal{F}_t^0$-adapted $\mathbb{R}^n$ vector-valued process, $\eta_0$ is an $\mathcal{F}_T^0$-measurable $\mathbb{R}^n$ vector-valued random variable, and $Q(\cdot)$, $R(\cdot)$, $H(\cdot)$ are symmetric.
    $u_{-i}(\cdot) := (u_1(\cdot),\cdots,u_{i-1}(\cdot),u_{i+1}(\cdot),\cdots,u_N(\cdot))$.

    The following remarks demonstrate the reasonableness of the model setup.

    \begin{Remark}
        In particular, for problems rooted in economics, it is usually common to discount future losses.
        This provides an accurate measure of the present value of all future losses.
        In such scenarios, the $i$th player, $1 \le i \le N$, is assumed to aim at minimizing the cost functional
        \begin{equation*}
            \begin{aligned}
                J^N_i(v_i(\cdot), v_{-i}(\cdot)) &= \mathbb{E} \bigg\{\int_0^T e^{-\rho t}\left[\big|\big|y_i(t) - \varGamma(t) y^{(N)}(t) - \tilde{\eta}(t)\big|\big|^2_{Q(t)} + \big|\big|v_i(t)\big|\big|^2_{R(t)} \right]dt\\
                &\qquad\qquad + e^{-\rho T}\big|\big|y_i(T) - \varGamma_0 y^{(N)}(T) - \tilde{\eta}_0\big|\big|^2_H\bigg\},
            \end{aligned}
        \end{equation*}
        where $\rho$ being the discount factor, subjecting to the state equation
        \begin{equation*}
            \left\{
            \begin{aligned}
                dy_i(t) =& \big[A(t)y_i(t) + B(t)v_i(t) + G(t)y^{(N)}(t) + \tilde{f}(t)\big]dt\\
                 & + \big[C(t)y_i(t) + D(t)v_i(t) + F(t)y^{(N)}(t) + \tilde{\sigma}(t)\big]dW_i(t)\\
                 & + \big[C_0(t)y_i(t) + D_0(t)v_i(t) + F_0(t)y^{(N)}(t) + \tilde{\sigma}_0(t)\big]dW_0(t),\\
                y_i(0) =&\ \xi_i .
            \end{aligned}
            \right.
        \end{equation*}
        Introducing $x_i(t) := e^{-\frac{1}{2}\rho t} y_i(t)$, $u_i(t) := e^{-\frac{1}{2}\rho t} v_i(t)$, $1 \le i \le N$, $\eta(t) := e^{-\frac{1}{2}\rho t} \tilde{\eta}(t)$, $f(t) := e^{-\frac{1}{2}\rho t} \tilde{f}(t)$, $\sigma(t) := e^{-\frac{1}{2}\rho t} \tilde{\sigma}(t)$, $\eta_0 := e^{-\frac{1}{2}\rho T} \tilde{\eta}_0$, this problem can be rephrased as
        \begin{equation*}
            \begin{aligned}
                J^N_i(v_i(\cdot), v_{-i}(\cdot)) &= \mathbb{E} \bigg\{\int_0^T \Big[\big|\big|e^{-\frac{1}{2}\rho t} y_i(t) - \varGamma(t) e^{-\frac{1}{2}\rho t} y^{(N)}(t) - e^{-\frac{1}{2}\rho t} \tilde{\eta}(t)\big|\big|^2_{Q(t)}  \\
                &\qquad\quad + \big|\big|e^{-\frac{1}{2}\rho t} v_i(t)\big|\big|^2_{R(t)} \Big]dt\\
                &\qquad + \big|\big|e^{-\frac{1}{2}\rho T} y_i(T) - \varGamma_0 e^{-\frac{1}{2}\rho T} y^{(N)}(T) - e^{-\frac{1}{2}\rho T}\tilde{\eta}_0\big|\big|^2_H\bigg\}\\
                &= \mathbb{E} \bigg\{\int_0^T \Big[\big|\big|x_i(t) - \varGamma(t) x^{(N)}(t) - \eta(t)\big|\big|^2_{Q(t)} + \big|\big|u_i(t)\big|\big|^2_{R(t)} \Big]dt \\
                &\qquad\qquad + \big|\big|x_i(T) - \varGamma_0 x^{(N)}(T) - \eta_0\big|\big|^2_H\bigg\} ,
            \end{aligned}
        \end{equation*}
        where
        \begin{equation*}
            \left\{
            \begin{aligned}
                dx_i(t) =&\ d(e^{-\frac{1}{2}\rho t} y_i(t))= -\frac{1}{2}\rho e^{-\frac{1}{2}\rho t} y_i(t)dt + e^{-\frac{1}{2}\rho t} dy_i(t)\\
                =& \bigg[\left(A(t) - \frac{1}{2}\rho\right) x_i(t) + B(t)u_i(t) + G(t)x^{(N)}(t) + f(t)\bigg]dt\\
                & + \big[C(t)x_i(t) + D(t)u_i(t) + F(t)x^{(N)}(t) + \sigma(t)\big]dW_i(t)\\
                & + \big[C_0(t)x_i(t) + D_0(t)u_i(t) + F_0(t)x^{(N)}(t) + \sigma_0(t)\big]dW_0(t),\\
                x_i(0) =&\ \xi_i .
            \end{aligned}
            \right.
        \end{equation*}
        So discounting future losses by the factor $\rho$ can be incorporated into our framework merely by substituting matrix $A(t)$ by $A(t) - \frac{1}{2}\rho$.
    \end{Remark}

    \begin{Remark}
        Consider the case where cross terms are included in the running cost of the cost functional.
        In such scenarios, the $i$th player, $1 \le i \le N$, is assumed to aim at minimizing the cost functional
        \begin{equation}\label{remark cost1}
            \begin{aligned}
                J^N_i(v_i(\cdot), v_{-i}(\cdot)) =&\ \mathbb{E} \int_0^T \bigg[\big|\big|x_i(t) - \varGamma(t) x^{(N)}(t) - \eta(t)\big|\big|^2_{Q(t)} + \big|\big|v_i(t)\big|\big|^2_{R(t)} \\
                &\qquad + 2\big(x_i(t) - \varGamma(t) x^{(N)}(t) - \eta(t)\big)^\top S(t) v_i(t) \bigg]dt,
            \end{aligned}
        \end{equation}
        subject to the  state equation
        \begin{equation}\label{remark state1}
            \left\{
            \begin{aligned}
                dx_i(t) =& \big[A(t)x_i(t) + B(t)v_i(t) + G(t)x^{(N)}(t) + f(t)\big]dt\\
                 & + \big[C(t)x_i(t) + D(t)v_i(t) + F(t)x^{(N)}(t) + \sigma(t)\big]dW_i(t)\\
                 & + \big[C_0(t)x_i(t) + D_0(t)v_i(t) + F_0(t)x^{(N)}(t) + \sigma_0(t)\big]dW_0(t),\\
                x_i(0) =&\ \xi_i .
            \end{aligned}
            \right.
        \end{equation}
        Introducing $u_i(t) := v_i(t) + R^{-1}(t) S^\top(t)\big(x_i(t) - \varGamma(t) x^{(N)}(t) - \eta(t)\big)$, $1 \le i \le N$, we can rewrite this problem as
        \begin{equation}\label{remark cost2}
            \begin{aligned}
                &J^N_i(v_i(\cdot), v_{-i}(\cdot)) = \mathbb{E} \int_0^T \bigg[\big|\big|v_i(t) + R^{-1}(t) S^\top(t)(x_i(t) - \varGamma(t) x^{(N)}(t) - \eta(t))\big|\big|^2_{R(t)}\\
                &\qquad\qquad\qquad\qquad + \big|\big|x_i(t) - \varGamma(t) x^{(N)}(t) - \eta(t)\big|\big|^2_{Q(t) - S(t) R^{-1}(t) S^\top(t)} \bigg]dt \\
                =&\ \mathbb{E} \int_0^T \bigg[\big|\big|x_i(t) - \varGamma(t) x^{(N)}(t) - \eta(t)\big|\big|^2_{Q(t) - S(t) R^{-1}(t) S^\top(t)} + \big|\big|u_i(t)\big|\big|^2_{R(t)}\bigg]dt,
            \end{aligned}
        \end{equation}
        where
        \begin{equation}\label{remark state2}
            \left\{
            \begin{aligned}
                dx_i(t) =& \big[(A(t) - B(t) R^{-1}(t) S^\top(t))x_i(t) + B(t)u_i(t) + (G(t) + B(t) R^{-1}(t)  \\
                 &\quad \times S^\top(t) \varGamma(t))x^{(N)}(t)+ (f(t) + B(t) R^{-1}(t) S^\top(t) \eta(t))\big]dt\\
                 & + \big[(C(t) - D(t) R^{-1}(t) S^\top(t))x_i(t) + D(t)u_i(t) + (F(t) + D(t) R^{-1}(t)  \\
                 &\quad \times S^\top(t) \varGamma(t))x^{(N)}(t)+ (\sigma(t) + D(t) R^{-1}(t) S^\top(t) \eta(t))\big]dW_i(t)\\
                 & + \big[(C_0(t) - D_0(t) R^{-1}(t) S^\top(t))x_i(t) + D_0(t)u_i(t) + (F_0(t) + D_0(t) R^{-1}(t)  \\
                 &\quad \times S^\top(t) \varGamma(t))x^{(N)}(t)+ (\sigma_0(t) + D_0(t) R^{-1}(t) S^\top(t) \eta(t))\big]dW_0(t),\\
                x_i(0) =&\ \xi_i .
            \end{aligned}
            \right.
        \end{equation}
        By determining the equilibrium strategies for problems (\ref{remark cost2}) and (\ref{remark state2}), one can thereby identify the equilibrium strategies for problems (\ref{remark cost1}) and (\ref{remark state1}).
    \end{Remark}

    Let $\mathcal{F}_t$ be the $\sigma$-algebra generated by $\{\xi_i, W_i(s), W_0(s), s \le t, 1 \le i \le N\}$.
    Denote $\mathcal{F}_t^i$ be the $\sigma$-algebra generated by $\{x_i(s), W_i(s), W_0(s), s \le t\}, 1 \le i \le N$.

    We define the centralized control set as
    \begin{equation*}
        \begin{aligned}
        \mathscr{U}_c[0,T] &:= \big\{(u_1(\cdot),\cdots,u_N(\cdot))|u_i(\cdot) \in L^2_{\mathcal{F}}(0,T;\mathbb{R}), 1 \le i \le N\big\},
        \end{aligned}
    \end{equation*}
    and the decentralized control set as
    \begin{equation*}
        \begin{aligned}
        \mathscr{U}_d[0,T] &:= \big\{(u_1(\cdot),\cdots,u_N(\cdot))|u_i(\cdot) \in L^2_{\mathcal{F}^i}(0,T;\mathbb{R}), 1 \le i \le N\big\}.
        \end{aligned}
    \end{equation*}

    We summarize and supplement the conditions of the above coefficients in the following assumption.

    \noindent {\bf (A1)} $\{\xi_i\}, i = 1,2,\cdots,N$ are a sequence of i.i.d. random variables, with $\mathbb{E}[\xi_i] = \bar{\xi}$, $i = 1,2,\cdots,N$, and there exists a constant $c$ such that $\sup_{1 \le i \le N} \mathbb{E} [||\xi_i||^2] \le c.$

    \noindent {\bf (A2)} $\{W_i(t), 0 \le i \le N \}$ are mutually independent, which are also independent of $\{\xi_i, 1 \le i \le N \}$.

    \noindent {\bf (A3)}
    (i) $A(\cdot)$, $G(\cdot)$, $C(\cdot)$, $F(\cdot)$, $C_0(\cdot)$, $F_0(\cdot)$, $\varGamma(\cdot) \in L^\infty(0,T;\mathbb{R}^{n \times n})$, and $B(\cdot)$, $D(\cdot)$, $D_0(\cdot) \in L^\infty(0,T;\mathbb{R}^{n \times r})$;

    (ii) $Q(\cdot) \in L^\infty(0,T;\mathbb{S}^n)$, $R(\cdot) \in L^\infty(0,T;\mathbb{S}^r)$;

    (iii) $H \in \mathbb{S}^n$ and $\varGamma_0 \in \mathbb{R}^{n \times n}$ are bounded;

    (iv) $f(\cdot)$, $\sigma(\cdot)$, $\sigma_0(\cdot)$, $\eta(\cdot) \in L^2_{\mathcal{F}^0}(0,T;\mathbb{R}^n)$; $\eta_0 \in \mathbb{R}^n$ is $\mathcal{F}_T^0$-measurable, and $\mathbb{E}[||\eta_0||^2] < \infty$.\\

    In this paper, we explore the following problem:

    \noindent {\bf (PD)}: Find an $\epsilon$-Nash equilibrium strategy $u^{*N}(\cdot) := (u_1^*(\cdot),\cdots,u_N^*(\cdot))$ among $u^N(\cdot) \in \mathscr{U}_{d}[0,T]$ for (\ref{finite cost}), subject to (\ref{state}).

    Building on the models in \cite{Fershtman-Kamien-87}, \cite{Wiszniewska-Bodnar-Mirota-15}, \cite{Wang-25}, we present a production output adjustment problem under noisy sticky prices as a case study.
    This illuminates the motivation and practical context of problem {\bf (PD)}.

    \begin{example}
        Consider a large market with $N$ firms, assuming prices are sticky and subject to random noise.
        Each firm's production output is given by
        \begin{equation*}
            dq_i(t) = \left(-\mu q_i(t) + b u_i(t)\right)dt + \left(\rho q_i(t) + \varrho u_i(t)\right)d\omega_i(t), \quad q_i(0) = \xi_i, \quad i = 1,\cdots,N,
        \end{equation*}
        where $\{ \omega_i(\cdot), i = 1,\cdots,N \}$ are independent standard Brownian motions, which are also independent of initial outputs of all firms $\{ \xi_i, i = 1,\cdots,N \}$.
        Here, $\mu$ represents the output adjustment friction, $\rho$ indicates that large-scale production increases supply chain volatility, and $\varrho$ reflects the fluctuation risk associated with production adjustments.
        Assume that the noisy prices in the following form
        \begin{equation*}
            dp(t) = \alpha\left(\beta - q^{(N)}(t) + p(t)\right)dt + \left(\sigma p(t) + \varsigma q^{(N)}(t)\right)d\omega_0(t), \quad p(0) = \zeta,
        \end{equation*}
        where $\alpha > 0$ being the speed of adjustment to the level on the demand function, and $\omega_0(t)$ is a standard Brownian motion independent of $\{ \omega_i(\cdot)$, $i = 1,\cdots,N \}$.
        $q^{(N)}(t) = \frac{1}{N} \sum_{i=1}^{N} q_i(t)$ is the average of firms' outputs, and $\beta - q^{(N)}(t)$ is the price on the demand function for the given production level of firms.
        $\sigma$ reflects more intense fluctuations at high prices, and $\varsigma$ reflects higher price volatility with increasing market supply.
        Each firm's cost functional is
        \begin{equation*}
            J^N_i(u_i(\cdot), u_{-i}(\cdot)) = \mathbb{E} \left\{\int_0^T \left(-p(t)q_i(t) + cq_i + ru_i^2\right)dt + \gamma \left(c_T - p(T)\right)q_i(T)\right\}, \quad i = 1,\cdots,N,
        \end{equation*}
        where $r > 0$ is the adjustment cost rate, $\gamma > 0$ the terminal cost rate, $0 < c < \beta$ ensures the output in the steady state is positive, and $c_T > 0$.
        Each firm aims to minimize its cost over $\mathscr{U}_d^i[0,T] := \big\{u_i(\cdot)|u_i(\cdot) \text{ is adapted to } \sigma\{ \xi_i, \zeta, \omega_i(s), \omega_0(s), s \le t \}, \mathbb{E}\int_0^T u_i^2(t) < \infty\}, i = 1,\cdots,N$.
        Evidently, this problem can be construed as problem {\bf (PD)}.
    \end{example}

    \section{Main results}
    In this section, we will first address the finite-horizon problem, and then extend the results to the infinite-horizon problem.

    \subsection{The finite-horizon problem}

    Subsequently, the time arguments of functions may be omitted if their exclusion does not lead to confusion.
    Let's first consider the centralization issue below.

    \noindent {\bf (PC)}: Find a Nash equilibrium strategy $\check{u}^{N}(\cdot) := (\check{u}_1(\cdot),\cdots,\check{u}_N(\cdot))$ among $u^N(\cdot) \in \mathscr{U}_{c}[0,T]$ for (\ref{finite cost}), subject to (\ref{state}).

    We initially derive the following result.

    \begin{mythm}\label{result1}
        Under Assumptions (A1)-(A3), for the initial value $\xi_i, i = 1,\cdots,N$, problem {\bf (PC)} admits a Nash equilibrium strategy $\check{u}^N(\cdot) \in \mathscr{U}_c[0,T]$, if and only if the following two conditions hold:

        (i) The adapted solution $(\check{x}_i(\cdot), \check{p}_i^j(\cdot), \check{q}_i^{j,k}(\cdot), i,j = 1,\cdots,N, k = 0,1,\cdots,N)$ to the FBSDEs
        \begin{equation}\label{adjoint FBSDE}
            \left\{
            \begin{aligned}
                d\check{x}_i &= \big[A\check{x}_i + B\check{u}_i + G\check{x}^{(N)} + f\big]dt + \big[C\check{x}_i + D\check{u}_i + F\check{x}^{(N)} + \sigma\big]dW_i \\
                &\quad + \big[C_0\check{x}_i + D_0\check{u}_i + F_0\check{x}^{(N)} + \sigma_0\big]dW_0,\\
                d\check{p}_i^i &= -\bigg[A^\top\check{p}_i^i + G^\top\check{p}_i^{(N)} + C^\top\check{q}_i^{i,i} + F^\top\check{q}_i^{(N)} + C_0^\top\check{q}_i^{i,0} + F_0^\top\check{q}_i^{(N),0}  \\
                &\qquad + \left(I - \frac{\varGamma}{N}\right)^\top Q \left(\check{x}_i - \varGamma \check{x}^{(N)} - \eta\right)\bigg]dt + \sum_{k=1}^{N} \check{q}_i^{i,k} dW_k + \check{q}_i^{i,0} dW_0,\\
                d\check{p}_i^j &= -\bigg[A^\top\check{p}_i^j + G^\top\check{p}_i^{(N)} + C^\top\check{q}_i^{j,j} + F^\top\check{q}_i^{(N)} + C_0^\top\check{q}_i^{j,0} + F_0^\top\check{q}_i^{(N),0} \\
                &\qquad - \left. \frac{\varGamma}{N}^\top Q \left(\check{x}_i - \varGamma \check{x}^{(N)} - \eta\right)\right]dt + \sum_{k=1}^{N} \check{q}_i^{j,k} dW_k + \check{q}_i^{j,0} dW_0, \quad j \ne i,\\
                \check{x}_i(0) &= \xi_i , \quad \check{p}_i^i(T) = \left(I - \frac{\varGamma_0}{N}\right)^\top H \left(\check{x}_i(T) - \varGamma_0 \check{x}^{(N)}(T) - \eta_0\right) ,\\
                \check{p}_i^j(T) &= - \frac{\varGamma_0}{N}^\top H \left(\check{x}_i(T) - \varGamma_0 \check{x}^{(N)}(T) - \eta_0\right),
            \end{aligned}
            \right.
        \end{equation}
        where $\check{p}_i^{(N)} := \frac{1}{N} \sum_{j=1}^{N} \check{p}_i^j$, $\check{q}_i^{(N)} := \frac{1}{N} \sum_{j=1}^{N} \check{q}_i^{j,j}$, $\check{q}_i^{(N),0} := \frac{1}{N} \sum_{j=1}^{N} \check{q}_i^{j,0}$, satisfies the following stationarity condition:
        \begin{equation}\label{stationarity condition}
            B^\top\check{p}_i^i + D^\top\check{q}_i^{i,i} + D_0^\top\check{q}_i^{i,0} + R\check{u}_i = 0, \quad i = 1,\cdots,N.
        \end{equation}

        (ii) For $i = 1,\cdots,N$, the following convexity condition holds:
        \begin{equation}\label{convexity condition}
            \begin{aligned}
            \mathbb{E} \bigg\{\int_0^T \Big[\big|\big|\tilde{x}_i^i - \varGamma \tilde{x}^{i,(N)}\big|\big|^2_Q + \big|\big|u_i\big|\big|^2_R\Big]dt + \big|\big|\tilde{x}_i^i(T) - \varGamma_0 \tilde{x}^{i,(N)}(T)\big|\big|^2_H \bigg\} \ge 0,\\ \forall u_i(\cdot) \in \mathscr{U}_c^i[0,T],
            \end{aligned}
        \end{equation}
        where $\mathscr{U}_c^i[0,T] := \big\{u_i(\cdot)|u_i(\cdot) \in L^2_{\mathcal{F}}(0,T;\mathbb{R})\big\}$, and $\tilde{x}_i^i(\cdot)$, $\tilde{x}_j^i(\cdot)$ represent the solutions to the following SDEs:
        \begin{equation}\label{tilde x_i}
            \left\{
            \begin{aligned}
                d\tilde{x}_i^i =& \left[A\tilde{x}_i^i + Bu_i + G\tilde{x}^{i,(N)}\right]dt + \left[C\tilde{x}_i^i + Du_i + F\tilde{x}^{i,(N)}\right]dW_i \\
                &+ \left[C_0\tilde{x}_i^i + D_0u_i + F_0\tilde{x}^{i,(N)}\right]dW_0,\\
                \tilde{x}_i^i(0) =&\ 0,
            \end{aligned}
            \right.
        \end{equation}
        \begin{equation}\label{tilde x_j}
            \left\{
            \begin{aligned}
                d\tilde{x}_j^i =& \left[A\tilde{x}_j^i + G\tilde{x}^{i,(N)}\right]dt + \left[C\tilde{x}_j^i + F\tilde{x}^{i,(N)}\right]dW_i + \left[C_0\tilde{x}_j^i + F_0\tilde{x}^{i,(N)}\right]dW_0,\\
                \tilde{x}_j^i(0) =&\ 0, \qquad j \ne i,
            \end{aligned}
            \right.
        \end{equation}
        respectively, and $\tilde{x}^{i,(N)} := \frac{1}{N} \sum_{j=1}^{N} \tilde{x}_j^i$.
    \end{mythm}

    \begin{proof}
        We suppose that $(\check{u}_1(\cdot), \cdots, \check{u}_N(\cdot))$ is a candidate of the optimal strategy, $(\check{x}_1(\cdot), \cdots, \check{x}_N(\cdot))$ is the corresponding state process.
        For any $u_i(\cdot) \in \mathscr{U}_c^i[0,T]$ and $\theta \in \mathbb{R}$, let $x_j^{i,\theta}(\cdot), j=1,\cdots,N$ be the solutions to the following perturbed state equations:
        \begin{equation*}
            \left\{
            \begin{aligned}
                dx_i^{i,\theta} =& \big[Ax_i^{i,\theta} + B(\check{u}_i + \theta u_i) + Gx^{i,\theta,(N)} + f\big]dt + \big[Cx_i^{i,\theta} + D(\check{u}_i + \theta u_i) + \sigma\big]dW_i \\
                &+ \big[C_0x_i^{i,\theta} + D_0(\check{u}_i + \theta u_i) + \sigma_0\big]dW_0,\\
                x_i^{i,\theta}(0) =&\ \xi_i,
            \end{aligned}
            \right.
        \end{equation*}
        \begin{equation*}
            \left\{
            \begin{aligned}
                dx_j^{i,\theta} =& \big[Ax_j^{i,\theta} + B\check{u}_j + Gx^{i,\theta,(N)} + f\big]dt + \big[Cx_j^{i,\theta} + D\check{u}_j + \sigma\big]dW_i \\
                &+ \big[C_0x_j^{i,\theta} + D_0\check{u}_j + \sigma_0\big]dW_0, \qquad j \ne i,\\
                x_j^{i,\theta}(0) =&\ \xi_j,
            \end{aligned}
            \right.
        \end{equation*}
        where $x^{i,\theta,(N)} := \frac{1}{N} \sum_{j=1}^{N} x_j^{i,\theta}$.
        Here, the superscript $i$ denotes the state corresponding to the perturbation of the $i$th control $u_i(\cdot)$.

        Then, $\tilde{x}_j^i(\cdot) := \frac{x_j^{i,\theta}(\cdot) - \check{x}_j^i(\cdot)}{\theta}, j = 1,\cdots,N$ are independent of $\theta$ and satisfy (\ref{tilde x_i}) and (\ref{tilde x_j}).
        Applying It\^{o}'s formula to $\langle \check{p}_i^j(\cdot), \tilde{x}_j^i(\cdot) \rangle, j = 1,\cdots,N$, integrating from $0$ to $T$, and taking the expectation, we have
        \begin{equation*}
            \begin{aligned}
                &\mathbb{E} \bigg[ \left(\check{x}_i(T) - \varGamma_0 \check{x}^{(N)}(T) - \eta_0\right)^\top H \left(I - \frac{\varGamma_0}{N}\right) \tilde{x}_i^i(T)\bigg] \\
                = &\ \mathbb{E} \big[\langle \check{p}_i^i(T), \tilde{x}_i^i(T) \rangle - \langle \check{p}_i^i(0) \tilde{x}_i^i(0)\big \rangle\big] \\
                = &\ \mathbb{E} \int_0^T \bigg\{ -\bigg[A^\top\check{p}_i^i + G^\top\check{p}_i^{(N)} + C^\top\check{q}_i^{i,i} + F^\top\check{q}_i^{(N)} + C_0^\top\check{q}_i^{i,0} + F_0^\top\check{q}_i^{(N),0} \\
                &\quad\qquad + \left(I - \frac{\varGamma}{N}\right)^\top Q \left(\check{x}_i - \varGamma \check{x}^{(N)} - \eta\right)\bigg]^\top \tilde{x}_i^i
                 + \check{p}_i^{i\top} \left[A\tilde{x}_i^i + Bu_i + G\tilde{x}^{i,(N)}\right]  \\
                &\quad\qquad + \check{q}_i^{i,i\top} \left[C\tilde{x}_i^i + Du_i + F\tilde{x}^{i,(N)}\right]+ \check{q}_i^{i,0\top} \left[C_0\tilde{x}_i^i + D_0u_i + F_0\tilde{x}^{i,(N)}\right] \bigg\}dt,
            \end{aligned}
        \end{equation*}
        \begin{equation*}
            \begin{aligned}
                &\mathbb{E} \bigg[ -\left(\check{x}_i(T) - \varGamma_0 \check{x}^{(N)}(T) - \eta_0\right)^\top H \frac{\varGamma_0}{N} \tilde{x}_j^i(T)\bigg] \\
                = &\ \mathbb{E} \big[\langle \check{p}_i^j(T), \tilde{x}_j^i(T) \rangle - \langle \check{p}_i^j(0) \tilde{x}_j^i(0) \rangle\big] \\
                = &\ \mathbb{E} \int_0^T \bigg\{ -\bigg[A^\top\check{p}_i^j + G^\top\check{p}_i^{(N)} + C^\top\check{q}_i^{j,j} + F^\top\check{q}_i^{(N)} + C_0^\top\check{q}_i^{j,0} + F_0^\top\check{q}_i^{(N),0} \\
                &\quad\qquad  - \frac{\varGamma}{N}^\top Q \left(\check{x}_i - \varGamma \check{x}^{(N)} - \eta\right)\bigg]^\top \tilde{x}_j^i + \check{p}_i^{j\top} \left[A\tilde{x}_j^i + G\tilde{x}^{i,(N)}\right]\\
                &\quad\qquad  + \check{q}_i^{j,j\top} \left[C\tilde{x}_j^i + F\tilde{x}^{i,(N)}\right] + \check{q}_i^{j,0\top} \left[C_0\tilde{x}_j^i + F_0\tilde{x}^{i,(N)}\right] \bigg\}dt, \qquad j \ne i,
            \end{aligned}
        \end{equation*}
        which further implies that
        \begin{equation*}
            \begin{aligned}
                &\mathbb{E} \bigg[ \left(\check{x}_i(T) - \varGamma_0 \check{x}^{(N)}(T) - \eta_0\right)^\top H \left(\tilde{x}_i^i(T) - \varGamma_0 \tilde{x}^{i,(N)}(T)\right) \bigg] \\
                = &\sum_{j=1}^{N} \mathbb{E} \big[\langle \check{p}_i^j(T), \tilde{x}_j^i(T) \rangle - \langle \check{p}_i^j(0) \tilde{x}_j^i(0) \rangle\big] \\
                = &\ \mathbb{E} \int_0^T \left[ \left(\check{p}_i^{i\top}B + \check{q}_i^{i,i\top}D + \check{q}_i^{i,0\top}D_0\right) u_i + \left(\check{x}_i - \varGamma \check{x}^{(N)} - \eta\right)^\top Q \left(\tilde{x}_i^i - \varGamma \tilde{x}^{i,(N)}\right) \right]dt.
            \end{aligned}
        \end{equation*}
        Therefore,
        \begin{equation*}
        \begin{aligned}
            &J^N_i(\check{u}_i(\cdot) + \theta u_i(\cdot), \check{u}_{-i}(\cdot)) - J^N_i(\check{u}_i(\cdot), \check{u}_{-i}(\cdot))\\
            =&\ \mathbb{E} \left\{\int_0^T \left[\big|\big|x_i^{i,\theta} - \varGamma x^{i,\theta,(N)} - \eta\big|\big|^2_Q + \big|\big|\check{u}_i + \theta u_i\big|big|^2_R\right]dt
            + \big|\big|x_i^{i,\theta}(T) - \varGamma_0 x^{i,\theta,(N)}(T) - \eta_0\big|\big|^2_H\right\}\\
            &- \mathbb{E} \left\{\int_0^T \left[\big|\big|\check{x}_i - \varGamma \check{x}^{(N)} - \eta\big|\big|^2_Q +\big|\big|\check{u}_i\big|\big|^2_R\right]dt
            + \big|\big|\check{x}_i(T) - \varGamma_0 \check{x}^{(N)}(T) - \eta_0\big|\big|^2_H\right\} \\
            =&\ \theta^2 \mathbb{E} \left\{\int_0^T \left[\big|\big|\tilde{x}_i^i - \varGamma \tilde{x}^{i,(N)}\big|\big|^2_Q + \big|\big|u_i\big|\big|^2_R\right]dt + \big|\big|\tilde{x}_i^i(T) - \varGamma_0 \tilde{x}^{i,(N)}(T)\big|\big|^2_H\right\} \\
            & + 2\theta \mathbb{E} \bigg\{\int_0^T \bigg[\left(\tilde{x}_i^i - \varGamma \tilde{x}^{i,(N)}\right)^\top Q \left(\check{x}_i - \varGamma \check{x}^{(N)} - \eta\right) + u_i^\top R \check{u}_i\bigg]dt \\
            & \qquad + \left(\tilde{x}_i^i(T) - \varGamma_0 \tilde{x}^{i,(N)}(T)\right)^\top H \left(\check{x}_i(T) - \varGamma_0 \check{x}^{(N)}(T) - \eta_0\right)\bigg\} \\
            =&\ \theta^2 \mathbb{E} \left\{\int_0^T \left[\big|\big|\tilde{x}_i^i - \varGamma \tilde{x}^{i,(N)}\big|\big|^2_Q + \big|\big|u_i\big|\big|^2_R\right]dt + \big|\big|\tilde{x}_i^i(T) - \varGamma_0 \tilde{x}^{i,(N)}(T)\big|\big|^2_H\right\} \\
            & + 2\theta \mathbb{E} \bigg\{\int_0^T u_i^\top \left(B^\top\check{p}_i^i + D^\top\check{q}_i^{i,i} + D_0^\top\check{q}_i^{i,0} + R\check{u}_i\right)dt \bigg\}.
        \end{aligned}
        \end{equation*}
        Thus, we have,
        $$J^N_i(\check{u}_i(\cdot), \check{u}_{-i}(\cdot)) \le J^N_i(\check{u}_i(\cdot) + \theta u_i(\cdot), \check{u}_{-i}(\cdot)),\quad \forall u_i(\cdot) \in \mathscr{U}_c^i[0,T],$$
        if and only if (\ref{stationarity condition}) and (\ref{convexity condition}) hold. The proof is complete.
    \end{proof}

    To facilitate the subsequent analysis, we introduce the following definitions and assumptions of convexity.

    \begin{mydef}
        Problem {\bf (PC)} is called convex, if for any $i = 1,\cdots,N$, $0 < \lambda < 1$ and $u_i(\cdot)$, $v_i(\cdot) \in \mathscr{U}_c^i[0,T]$,
        $$J^N_i(\lambda u_i(\cdot) + (1-\lambda) v_i(\cdot), u_{-i}(\cdot)) \le \lambda J^N_i(u_i(\cdot), u_{-i}(\cdot)) + (1-\lambda) J^N_i(v_i(\cdot), u_{-i}(\cdot)).$$
    \end{mydef}

    \noindent {\bf (A4)} Problem {\bf (PC)} is convex.

    \begin{mypro}
        Problem {\bf (PC)} is convex in $u_i(\cdot)$ if and only if (\ref{convexity condition}) holds.
    \end{mypro}
    \begin{proof}
        Let $y_i(\cdot)$ and $z_i(\cdot)$ be the state processes corresponding to control $v_i(\cdot)$ and $w_i(\cdot)$, respectively.
        For $0 < \lambda < 1$, we denote $\tilde{x}_i^i := y_i - z_i$, $u_i := v_i - w_i$, then
        \begin{equation*}
            \begin{aligned}
                \lambda J^N_i(v_i(\cdot), u_{-i}(\cdot)) + (1-\lambda) J^N_i(w_i(\cdot), u_{-i}(\cdot)) - J^N_i(\lambda v_i(\cdot) + (1-\lambda) w_i(\cdot), u_{-i}(\cdot)) \\
                = \lambda(1-\lambda) \mathbb{E} \left\{\int_0^T \left[\big|\big|\tilde{x}_i^i - \varGamma \tilde{x}^{i,(N)}\big|\big|^2_Q + \big|\big|u_i\big|\big|^2_R\right]dt + \big|\big|\tilde{x}_i^i(T) - \varGamma_0 \tilde{x}^{i,(N)}(T)\big|\big|^2_H\right\}.
            \end{aligned}
        \end{equation*}
        By the definition of convexity, the proposition follows.
    \end{proof}

    \begin{Remark}
        The above proposition shows that the convexity condition (\ref{convexity condition}) is equivalent to the convexity of the problem {\bf (PC)}.
        In particular, if the state weight and control weight $Q(\cdot), R(\cdot), H \ge 0$ in the cost functional, then the assumption holds.
    \end{Remark}

    Denote $\mathbb{E}_0[\cdot] := \mathbb{E}[\cdot|\mathcal{F}_t^0]$.
    By directly taking the limit as $N \to \infty$ and applying the law of large numbers, $\check{x}_i(\cdot)$, $\check{p}_i^i(\cdot)$, $\check{p}_i^j(\cdot)$, $\check{q}_i^{i,i}(\cdot)$, $\check{q}_i^{i,0}(\cdot)$, $\check{q}_i^{j,j}(\cdot)$, $\check{q}_i^{j,0}(\cdot)$, $\check{x}^{(N)}(\cdot)$, $\check{p}_i^{(N)}(\cdot)$, $\check{q}_i^{(N)}(\cdot)$, $\check{q}_i^{(N),0}(\cdot)$, $j \ne i, i = 1,\cdots,N$ in (\ref{adjoint FBSDE}) approximates to $\bar{x}_i(\cdot)$, $\bar{p}_i^i(\cdot)$, $\bar{p}_i^j(\cdot)$, $\bar{q}_i^{i,i}(\cdot)$, $\bar{q}_i^{i,0}(\cdot)$, $\bar{q}_i^{j,j}(\cdot)$, $\bar{q}_i^{j,0}(\cdot)$, $\mathbb{E}_0[\bar{x}_i(\cdot)]$, $\mathbb{E}_0[\bar{p}_i^j(\cdot)]$, $\mathbb{E}_0[\bar{q}_i^{j,j}(\cdot)]$, $\mathbb{E}_0[\bar{q}_i^{j,0}(\cdot)]$, $j \ne i, i = 1,\cdots,N$, respectively, as below:
    \begin{equation}\label{limit FBSDE}
        \left\{
        \begin{aligned}
            d\bar{x}_i &= \big[A\bar{x}_i + B\bar{u}_i + G\mathbb{E}_0[\bar{x}_i] + f\big]dt + \big[C\bar{x}_i + D\bar{u}_i + F\mathbb{E}_0[\bar{x}_i] + \sigma\big]dW_i \\
            &\quad + \big[C_0\bar{x}_i + D_0\bar{u}_i + F_0\mathbb{E}_0[\bar{x}_i] + \sigma_0\big]dW_0,\\
            d\bar{p}_i^i &= -\left[A^\top\bar{p}_i^i + G^\top\mathbb{E}_0[\bar{p}_i^j] + C^\top\bar{q}_i^{i,i} + F^\top\mathbb{E}_0[\bar{q}_i^{j,j}] + C_0^\top\bar{q}_i^{i,0} + F_0^\top\mathbb{E}_0[\bar{q}_i^{j,0}] \right. \\
            &\qquad + Q \left(\bar{x}_i - \varGamma \mathbb{E}_0[\bar{x}_i] - \eta\right)\Big]dt + \bar{q}_i^{i,i} dW_i + \bar{q}_i^{i,0} dW_0,\\
            d\bar{p}_i^j &= -\left[A^\top\bar{p}_i^j + G^\top\mathbb{E}_0[\bar{p}_i^j] + C^\top\bar{q}_i^{j,j} + F^\top\mathbb{E}_0[\bar{q}_i^{j,j}] + C_0^\top\bar{q}_i^{j,0} + F_0^\top\mathbb{E}_0[\bar{q}_i^{j,0}] \right]dt\\
            &\quad + \bar{q}_i^{j,j} dW_j + \bar{q}_i^{j,0} dW_0, \quad j \ne i,\\
            \bar{x}_i(0) &= \xi_i , \quad \bar{p}_i^i(T) = H \left(\bar{x}_i(T) - \varGamma_0 \mathbb{E}_0[\bar{x}_i(T)] - \eta_0\right), \bar{p}_i^j(T) = 0,
        \end{aligned}
        \right.
    \end{equation}
    with the decentralized stationarity condition:
    \begin{equation}\label{limit stationarity condition}
        B^\top\bar{p}_i^i + D^\top\bar{q}_i^{i,i} + D_0^\top\bar{q}_i^{i,0} + R\bar{u}_i = 0, \quad i = 1,\cdots,N,
    \end{equation}
    where
    \begin{equation}\label{E0barxi}
        \left\{
        \begin{aligned}
            d\mathbb{E}_0[\bar{x}_i] &= \big[\left(A + G\right)\mathbb{E}_0[\bar{x}_i] + B\mathbb{E}_0[\bar{u}_i] + f\big]dt + \big[\left(C_0 + F_0\right)\mathbb{E}_0[\bar{x}_i] + D_0\mathbb{E}_0[\bar{u}_i] + \sigma_0\big]dW_0,\\
            \mathbb{E}_0[\bar{x}_i(0)] &= \bar{\xi},\quad i = 1,\cdots,N.
        \end{aligned}
        \right.
    \end{equation}
    Note that $\mathbb{E}_0[\bar{u}_i]$ and $\mathbb{E}_0[\bar{x}_i]$ are not depend on $i$, we denote $\bar{x} := \mathbb{E}_0[\bar{x}_i]$, $\bar{u} := \mathbb{E}_0[\bar{u}_i]$.
    We can rewrite (\ref{E0barxi}) as
    \begin{equation}\label{barx}
        \left\{
        \begin{aligned}
            d\bar{x} &= \big[\left(A + G\right)\bar{x} + B\bar{u} + f\big]dt + \big[\left(C_0 + F_0\right)\bar{x} + D_0\bar{u} + \sigma_0\big]dW_0,\\
            \bar{x}(0) &= \bar{\xi}.
        \end{aligned}
        \right.
    \end{equation}

    From the last equation in (\ref{limit FBSDE}), and by invoking the existence and uniqueness of solutions to {\it backward stochastic differential equations} (BSDEs), we are pleasantly surprised to find that
    $$\bar{p}_i^j(\cdot) \equiv 0,\ \bar{q}_i^{j,j}(\cdot) \equiv 0,\ \bar{q}_i^{j,0}(\cdot) \equiv 0,\quad j \ne i,\ i = 1,\cdots,N.$$
    Accordingly, the second equation in (\ref{limit FBSDE}) can be rewritten as
    $$d\bar{p}_i^i = -\left[A^\top\bar{p}_i^i + C^\top\bar{q}_i^{i,i} + C_0^\top\bar{q}_i^{i,0} + Q \left(\bar{x}_i - \varGamma \bar{x} - \eta\right)\right]dt + \bar{q}_i^{i,i} dW_i + \bar{q}_i^{i,0} dW_0,\quad i = 1,\cdots,N.$$

    \begin{Remark}
        Due to the mean field term in the state equation (\ref{state}), there's correlation in individual states, causing a sharp increase in the number of adjoint variables.
        Using the conventional approach to decouple high-dimensional FBSDEs \cite{Si-Shi-25} would require four Riccati equations, complicating the analysis.
        To address this, before decoupling, we first obtain the limiting FBSDEs by taking the limit as $N \to \infty$ and applying the law of large numbers.
        From these, we unexpectedly find some variables are identically zero due to the existence and uniqueness of solutions to BSDEs.
        Thus, only two Riccati equations are needed in the subsequent decoupling process, greatly simplifying the problem solving process.
    \end{Remark}

    Consider the transformation
    \begin{equation}\label{barpii}
        \bar{p}_i^i(\cdot) = P(\cdot)\bar{x}_i(\cdot) + K(\cdot)\bar{x}(\cdot) + \varphi(\cdot),\quad i = 1,\cdots,N,
    \end{equation}
    where
    $d\varphi = \phi dt + \psi dW_0$, $P(T) = H$, $K(T) = -H\varGamma_0$, $\varphi(T) = -H\eta_0$. By It\^{o}'s formula, we get
    \begin{equation}\label{dbarpii}
        \begin{aligned}
            d\bar{p}_i^i &= \dot{P} \bar{x}_i dt + P\Big\{\big[A\bar{x}_i + B\bar{u}_i + G\bar{x} + f\big]dt + \big[C\bar{x}_i + D\bar{u}_i + F\bar{x} + \sigma\big]dW_i \\
            &\quad + \big[C_0\bar{x}_i + D_0\bar{u}_i + F_0\bar{x} + \sigma_0\big]dW_0\Big\} + \dot{K} \bar{x} dt + K\Big\{\big[\left(A+G\right)\bar{x} + B\bar{u} + f\big]dt\\
            &\quad  + \big[\left(C_0+F_0\right)\bar{x} + D_0\bar{u} + \sigma_0\big]dW_0\Big\} + \phi dt + \psi dW_0.
        \end{aligned}
    \end{equation}
    Comparing the coefficients of the corresponding diffusion terms in (\ref{dbarpii}) and the second equation of (\ref{limit FBSDE}), we have
    \begin{equation}\label{barqiii}
        \bar{q}_i^{i,i} = P\big[C\bar{x}_i + D\bar{u}_i + F\bar{x} + \sigma\big],\quad i = 1,\cdots,N,
    \end{equation}
    \begin{equation}\label{barqii0}
        \bar{q}_i^{i,0} = \psi + P\big[C_0\bar{x}_i + D_0\bar{u}_i + F_0\bar{x} + \sigma_0\big] + K\big[\left(C_0 + F_0\right)\bar{x} + D_0\bar{u} + \sigma_0\big],\quad i = 1,\cdots,N.
    \end{equation}

    Substitute (\ref{barpii}), (\ref{barqiii}) and (\ref{barqii0}) into (\ref{limit stationarity condition}), we have
    \begin{equation*}
        \begin{aligned}
            &B^\top \left[P\bar{x}_i + K\bar{x} + \varphi\right] + D^\top P\left[C\bar{x}_i + D\bar{u}_i + F\bar{x} + \sigma\right] + D_0^\top \big\{\psi + P\big[C_0\bar{x}_i + D_0\bar{u}_i \\
            & + F_0\bar{x}+ \sigma_0] + K\left[\left(C_0 + F_0\right)\bar{x} + D_0\bar{u} + \sigma_0\right]\big\} + R\bar{u}_i = 0,\quad i = 1,\cdots,N.
        \end{aligned}
    \end{equation*}
    Denote
    \begin{equation*}
        \left\{
        \begin{aligned}
            &\varUpsilon(P) := R + D^\top PD + D_0^\top PD_0,\quad \bar{\varUpsilon}(P,K) := \varUpsilon(P) + D_0^\top KD_0,\\
            &\varPhi(P) := B^\top P + D^\top PC + D_0^\top PC_0,\quad \bar{\varPhi}^\top(P,K) := \varPhi^\top(P) + KB + C_0^\top KD_0,\\
            &\varPsi(P,K) := B^\top K + D^\top PF + D_0^\top PF_0 + D_0^\top K\left(C_0 + F_0\right),\\
            &\varTheta(P,K,\varphi,\psi) := B^\top \varphi + D^\top P\sigma + D_0^\top \psi + D_0^\top P\sigma_0 + D_0^\top K\sigma_0.
        \end{aligned}
        \right.
    \end{equation*}
    In the above $\varUpsilon(P)$ emphasizes the dependence of $\varUpsilon$ on the matrix $P$, etc.

    If $$\mathcal{R}(B^\top) \cup \mathcal{R}(D^\top P) \cup \mathcal{R}(D_0^\top P) \subseteq \mathcal{R}(\varUpsilon(P)),$$
    $$\mathcal{R}(B^\top) \cup \mathcal{R}(D^\top P) \cup \mathcal{R}(D_0^\top (P+K)) \subseteq \mathcal{R}(\bar{\varUpsilon}(P,K)),$$
    then, we have
    \begin{equation}\label{barui}
        \bar{u}_i = -\varUpsilon^\dagger(P) \varPhi(P) \left(\bar{x}_i - \bar{x} \right) - \bar{\varUpsilon}^\dagger(P,K) \big[ (\varPhi(P) + \varPsi(P,K)) \bar{x} + \varTheta(P,K,\varphi,\psi) \big],\quad i = 1,\cdots,N,
    \end{equation}
    \begin{equation}\label{baru}
        \hspace{-6.5cm}\bar{u} = -\bar{\varUpsilon}^\dagger(P,K) \big[ (\varPhi(P) + \varPsi(P,K)) \bar{x} + \varTheta(P,K,\varphi,\psi) \big],
    \end{equation}
    where $M^\dagger$ denotes the Moore-Penrose pseudo inverse of a matrix $M$ (\cite{Penrose-55}).

    Continue comparing the coefficients of the drift terms, gives
    \begin{equation}\label{P}
        \left\{
        \begin{aligned}
            &\dot{P} + PA + A^\top P + C^\top PC + C_0^\top PC_0 - \varPhi^\top(P) \varUpsilon^\dagger(P) \varPhi(P) + Q = 0,\\
            &P(T) = H,
        \end{aligned}
        \right.
    \end{equation}
    \begin{equation}\label{K}
        \left\{
        \begin{aligned}
            &\dot{K} + K(A+G) + PG + A^\top K + \varPhi^\top(P) \varUpsilon^\dagger(P) \varPhi(P) - \bar{\varPhi}^\top(P,K) \bar{\varUpsilon}^\dagger(P,K) (\varPhi(P) + \varPsi(P,K))\\
            &\ + C^\top PF + C_0^\top PF_0 + C_0^\top K(C_0+F_0) - Q\varGamma = 0,\\
            &K(T) = - H\varGamma_0,
        \end{aligned}
        \right.
    \end{equation}
    \begin{equation}\label{varphi}
        \left\{
        \begin{aligned}
            d\varphi =& - \left[\left( A^\top - \bar{\varPhi}^\top(P,K) \bar{\varUpsilon}^\dagger(P,K) B^\top \right) \varphi + \left( C^\top - \bar{\varPhi}^\top(P,K) \bar{\varUpsilon}^\dagger(P,K) D^\top \right) P\sigma \right. \\
            &\quad + \left. \left( C_0^\top - \bar{\varPhi}^\top(P,K) \bar{\varUpsilon}^\dagger(P,K) D_0^\top \right)(\psi + (P+K) \sigma_0) + (P+K) f - Q\eta \right] dt + \psi dW_0,\\
            \varphi(T) =& -H\eta_0.
        \end{aligned}
        \right.
    \end{equation}

    For further analysis, we need to assume

    \noindent {\bf (A5)} Equations (\ref{P})-(\ref{varphi}) admit a set of solution $(P(\cdot), K(\cdot), \varphi(\cdot))$ such that $\varUpsilon(P) > 0$, $\bar{\varUpsilon}(P,K) > 0$, and $\mathcal{R}(B^\top) \cup \mathcal{R}(D^\top P) \cup \mathcal{R}(D_0^\top P) \subseteq \mathcal{R}(\varUpsilon(P))$, $\mathcal{R}(B^\top) \cup \mathcal{R}(D^\top P) \cup \mathcal{R}(D_0^\top (P+K)) \subseteq \mathcal{R}(\bar{\varUpsilon}(P,K))$.

    \begin{Remark}
        Note that (\ref{P}) is a symmetric Riccati differential equation, if it satisfies $Q(\cdot) \ge 0$, $\varUpsilon(P) > 0$, $H \ge 0$, then it admits a unique solution.
        (\ref{K}) is a non-symmetric Riccati equation, its solvability problem is challenging. For relevant details, refer to sources like \cite{Ma-Yong-99}, \cite{Freiling-Jank-Sarychev-00}, \cite{Freiling-02} and \cite{Kremer-Stefan-02}.
        If (\ref{P}) and (\ref{K}) admit solutions, then (\ref{varphi}) has a solution.
    \end{Remark}

    Inspired by the above discussion, we design
    \begin{equation}\label{hatui}
        \hat{u}_i = -\varUpsilon^\dagger(P) \varPhi(P) \left( \hat{x}_i - \bar{x} \right) - \bar{\varUpsilon}^\dagger(P,K) \big[ (\varPhi(P) + \varPsi(P,K)) \bar{x} + \varTheta(P,K,\varphi,\psi) \big],\quad i = 1,\cdots,N,
    \end{equation}
    where
    \begin{equation}\label{hatxi}
        \left\{
        \begin{aligned}
            d\hat{x}_i =& \Big[\left( A - B \varUpsilon^\dagger(P) \varPhi(P) \right) \hat{x}_i + G\hat{x}^{(N)} + \left( B \varUpsilon^\dagger(P) \varPhi(P) - B \bar{\varUpsilon}^\dagger(P,K)\right. \\
             &\quad \times(\varPhi(P) + \varPsi(P,K)) \Big)\bar{x}- B \bar{\varUpsilon}^\dagger(P,K) \varTheta(P,K,\varphi,\psi) + f\Big]dt \\
            &+ \Big[\left( C - D \varUpsilon^\dagger(P) \varPhi(P) \right) \hat{x}_i + F\hat{x}^{(N)} + \left( D \varUpsilon^\dagger(P) \varPhi(P) - D \bar{\varUpsilon}^\dagger(P,K)\right. \\
             &\quad \times(\varPhi(P) + \varPsi(P,K)) \Big)\bar{x}- D \bar{\varUpsilon}^\dagger(P,K) \varTheta(P,K,\varphi,\psi) + \sigma\Big]dW_i \\
            &+ \Big[\left( C_0 - D_0 \varUpsilon^\dagger(P) \varPhi(P) \right) \hat{x}_i + F_0\hat{x}^{(N)} + \left( D_0 \varUpsilon^\dagger(P) \varPhi(P) - D_0 \bar{\varUpsilon}^\dagger(P,K)\right.\\
             &\quad \times(\varPhi(P) + \varPsi(P,K)) \Big)\bar{x} - D_0 \bar{\varUpsilon}^\dagger(P,K) \varTheta(P,K,\varphi,\psi) + \sigma_0\Big]dW_0,\\
            \hat{x}_i(0) =&\ \xi_i,\quad i = 1,\cdots,N,
        \end{aligned}
        \right.
    \end{equation}
    \begin{equation}\label{barxu}
        \left\{
        \begin{aligned}
            d\bar{x} =& \left[\left( A + G - B \bar{\varUpsilon}^\dagger(P,K) (\varPhi(P) + \varPsi(P,K)) \right) \bar{x} - B \bar{\varUpsilon}^\dagger(P,K) \varTheta(P,K,\varphi,\psi) + f\right]dt \\
            & + \left[\left( C_0 + F_0 - D_0 \bar{\varUpsilon}^\dagger(P,K) (\varPhi(P) + \varPsi(P,K)) \right) \bar{x} - D_0 \bar{\varUpsilon}^\dagger(P,K) \varTheta(P,K,\varphi,\psi) + \sigma_0\right]dW_0, \\
            \bar{x}(0) =&\ \bar{\xi},
        \end{aligned}
        \right.
    \end{equation}
    and $\hat{x}^{(N)} := \frac{1}{N} \sum_{j=1}^{N} \hat{x}_i$.

    The following theorem shows the performance of the decentralized strategies.

    \begin{mythm}\label{NE}
        Assume that (A1)-(A5) hold. For problem {\bf (PD)}, $( \hat{u}_1(\cdot), \cdots, \hat{u}_N(\cdot))$ given in (\ref{hatui}) constitutes an $\epsilon$-Nash equilibrium, where $\epsilon = O\left(\frac{1}{\sqrt{N}}\right)$.
    \end{mythm}
    \begin{proof}
        By the Lemma A.2 in \cite{Wang-Huang-19}, we have for $u_i \in \mathcal{U}_c^i[0,T]$, if $J(u_i(\cdot), \hat{u}_{-i}(\cdot)) \le c_1$, then there exist an integer $N_0$ and a constant $c$ such that for $N \ge N_0$, $\mathbb{E}\int_0^T||u_i||^2 dt < c$, thus $\mathbb{E}\int_0^T||x_i||^2 dt < c$.
        Denote $\tilde{u}_i := u_i - \hat{u}_i$, $\tilde{x}_i := x_i - \hat{x}_i$, $\tilde{x}^{(N)} := \frac{1}{N}\sum_{j=1}^{N} \tilde{x}_i$.
        Then
        $$\mathbb{E}\int_0^T\left(||\tilde{x}_i||^2 + ||\tilde{u}_i||^2 \right)dt < \infty.$$
        It follows from (\ref{state}) and (\ref{hatxi}) that
        \begin{equation}
            \left\{
            \begin{aligned}
                d\tilde{x}_i =& \big[A\tilde{x}_i + B\tilde{u}_i + G\tilde{x}^{(N)}\big]dt + \big[C\tilde{x}_i + D\tilde{u}_i + F\tilde{x}^{(N)}\big]dW_i\\
                              & + \big[C_0\tilde{x}_i + D_0\tilde{u}_i + F_0\tilde{x}^{(N)}\big]dW_0,\\
                \tilde{x}_i(0) =&\ 0,
            \end{aligned}
            \right.
        \end{equation}
        \begin{equation}
            \left\{
            \begin{aligned}
                d\tilde{x}_j =& \left[A\tilde{x}_j + G\tilde{x}^{(N)}\right]dt + \left[C\tilde{x}_j + F\tilde{x}^{(N)}\right]dW_j + \left[C_0\tilde{x}_j + F_0\tilde{x}^{(N)}\right]dW_0,\\
                \tilde{x}_j(0) =&\ 0, \quad j \ne i,
            \end{aligned}
            \right.
        \end{equation}
        \begin{equation}
            \left\{
            \begin{aligned}\label{tildexN}
                d\tilde{x}^{(N)} =& \left[(A+G)\tilde{x}^{(N)} + \frac{B}{N}\tilde{u}_i\right]dt + \frac{1}{N}\sum_{j=1}^{N}\left[C\tilde{x}_j + F\tilde{x}^{(N)}\right]dW_j + \frac{D}{N}\tilde{u}_idW_i \\
                &+ \left[(C_0+F_0)\tilde{x}^{(N)} + \frac{D_0}{N}\tilde{u}_i\right]dW_0,\\
                \tilde{x}^{(N)}(0) =&\ 0,
            \end{aligned}
            \right.
        \end{equation}
        and
        \begin{equation*}
            \begin{aligned}
                &J^N_i(u_i(\cdot), \hat{u}_{-i}(\cdot)) - J^N_i(\hat{u}_i(\cdot), \hat{u}_{-i}(\cdot))\\
                =&\ \mathbb{E} \left\{\int_0^T \left[\big|\big|x_i - \varGamma x^{(N)} - \eta\big|\big|^2_Q + \big|\big|u_i\big|\big|^2_R\right]dt
                + \big|\big|x_i(T) - \varGamma_0 x^{(N)}(T) - \eta_0\big|\big|^2_H\right\}\\
                &\ - \mathbb{E} \left\{\int_0^T \left[\big|\big|\hat{x}_i - \varGamma \hat{x}^{(N)} - \eta\big|\big|^2_Q +\big|\big|\hat{u}_i\big|\big|^2_R\right]dt
                + \big|\big|\hat{x}_i(T) - \varGamma_0 \hat{x}^{(N)}(T) - \eta_0\big|\big|^2_H\right\} \\
                =&\ \mathbb{E} \left\{\int_0^T \left[\big|\big|\tilde{x}_i - \varGamma \tilde{x}^{(N)}\big|\big|^2_Q + \big|\big|\tilde{u}_i\big|\big|^2_R\right]dt + \big|\big|\tilde{x}_i(T) - \varGamma_0 \tilde{x}^{(N)}(T)\big|\big|^2_H\right\} \\
                & + 2 \mathbb{E} \bigg\{\int_0^T \left[\left(\tilde{x}_i - \varGamma \tilde{x}^{(N)}\right)^\top Q \left(\hat{x}_i - \varGamma \hat{x}^{(N)} - \eta\right) + \tilde{u}_i^\top R \hat{u}_i\right]dt \\
                & \qquad + \left(\tilde{x}_i(T) - \varGamma_0 \tilde{x}^{(N)}(T)\right)^\top H \left(\hat{x}_i(T) - \varGamma_0 \hat{x}^{(N)}(T) - \eta_0\right)\bigg\} .
            \end{aligned}
        \end{equation*}
        Denote
        \begin{equation*}
            \tilde{J}_i^N(\tilde{u}_i(\cdot), \hat{u}_{-i}(\cdot)) := \mathbb{E} \left\{\int_0^T \left[\big|\big|\tilde{x}_i - \varGamma \tilde{x}^{(N)}\big|\big|^2_Q + \big|\big|\tilde{u}_i\big|\big|^2_R\right]dt + \big|\big|\tilde{x}_i(T) - \varGamma_0 \tilde{x}^{(N)}(T)\big|\big|^2_H\right\},
        \end{equation*}
        \begin{equation*}
            \begin{aligned}
                \mathcal{I}_i^N :=&\ 2 \mathbb{E} \bigg\{\int_0^T \left[\left(\tilde{x}_i - \varGamma \tilde{x}^{(N)}\right)^\top Q \left(\hat{x}_i - \varGamma \hat{x}^{(N)} - \eta\right) + \tilde{u}_i^\top R \hat{u}_i\right]dt \\
                &\quad + \left(\tilde{x}_i(T) - \varGamma_0 \tilde{x}^{(N)}(T)\right)^\top H \left(\hat{x}_i(T) - \varGamma_0 \hat{x}^{(N)}(T) - \eta_0\right)\bigg\}.
            \end{aligned}
        \end{equation*}
        We have
        \begin{equation}
            J^N_i(u_i(\cdot), \hat{u}_{-i}(\cdot)) - J^N_i(\hat{u}_i(\cdot), \hat{u}_{-i}(\cdot)) = \tilde{J}_i^N(\tilde{u}_i(\cdot), \hat{u}_{-i}(\cdot)) + \mathcal{I}_i^N.
        \end{equation}
        By assumption (A4), $\tilde{J}_i^N(\tilde{u}_i(\cdot), \hat{u}_{-i}(\cdot)) \ge 0$, we only need to show that $\mathcal{I}_i^N = O\left(\frac{1}{\sqrt{N}}\right)$.

        Applying It\^{o}'s formula to $\langle\tilde{x}_i(\cdot), \bar{p}_i^i(\cdot)\rangle$, we have
        \begin{equation*}
            \begin{aligned}
                & \mathbb{E} \left[ \tilde{x}_i^\top(T)H\left( \bar{x}_i(T) - \varGamma_0\bar{x}(T) - \eta_0 \right) \right] \\
                =&\ \mathbb{E} \int_0^T \left\{\left[ A\tilde{x}_i + B\tilde{u}_i + G\tilde{x}^{(N)} \right]^\top \bar{p}_i^i - \tilde{x}_i^\top\left[A^\top \bar{p}_i^i + C^\top \bar{q}_i^{i,i} + C_0^\top \bar{q}_i^{i,0} + Q\left(\bar{x}_i - \varGamma\bar{x} - \eta\right)\right] \right.\\
                &\qquad \left.+ \left[ C\tilde{x}_i + D\tilde{u}_i + F\tilde{x}^{(N)} \right]^\top \bar{q}_i^{i,i} + \left[ C_0\tilde{x}_i + D_0\tilde{u}_i + F_0\tilde{x}^{(N)} \right]^\top \bar{q}_i^{i,0} \right\}dt\\
            =&\ \mathbb{E} \int_0^T \left[ -\tilde{x}_i^\top Q \left(\bar{x}_i - \varGamma\bar{x} - \eta\right) - \tilde{u}_i^\top R \bar{u}_i + \tilde{x}^{(N)\top}\left( G^\top \bar{p}_i^i + F^\top \bar{q}_i^{i,i} + F_0^\top \bar{q}_i^{i,0} \right) \right]dt.
            \end{aligned}
        \end{equation*}
        Hence,
        \begin{equation*}
            \begin{aligned}
                \mathcal{I}_i^N =&\ 2 \mathbb{E} \left\{\int_0^T \left[ \tilde{x}^{(N)\top} \left( G^\top \bar{p}_i^i + F^\top \bar{q}_i^{i,i} + F_0^\top \bar{q}_i^{i,0} - \varGamma^\top Q \left(\hat{x}_i - \varGamma \hat{x}^{(N)} - \eta\right) \right) \right.\right.\\
                &\qquad \qquad + \left.\tilde{x}_i^\top Q \left(\left(\hat{x}_i - \bar{x}_i\right) - \varGamma \left( \hat{x}^{(N)} - \bar{x} \right) \right) + \tilde{u}_i^\top R \left( \hat{u}_i - \bar{u}_i \right) \right]dt \\
                &\qquad -\tilde{x}^{(N)\top}(T) \varGamma_0^\top H \left(\hat{x}_i(T) - \varGamma_0 \hat{x}^{(N)}(T) - \eta_0\right) \\
                &\qquad +\tilde{x}_i^\top(T) H \left(  \hat{x}_i(T) - \bar{x}_i(T)  - \varGamma_0 \left( \hat{x}^{(N)}(T) - \bar{x}(T) \right) \right) \bigg\}.
            \end{aligned}
        \end{equation*}
        From (\ref{tildexN}),
        \begin{equation*}
            \begin{aligned}
                \tilde{x}^{(N)}(t) &= \int_0^t e^{(A+G)(t-\tau)}\bigg\{ \frac{B}{N}\tilde{u}_i dt + \frac{1}{N}\sum_{j=1}^{N}\left[C\tilde{x}_j + F\tilde{x}^{(N)}\right]dW_j \\
                & \qquad + \frac{D}{N}\tilde{u}_idW_i + \left[(C_0+F_0)\tilde{x}^{(N)} + \frac{D_0}{N}\tilde{u}_i\right]dW_0 \bigg\},\\
               \sup_{0 \le t \le T} \mathbb{E}[\tilde{x}^{(N)}(t)]^2 &\le \frac{c}{N^2} \sup_{0 \le t \le T} \mathbb{E} \left[ \int_0^t \tilde{u}_i(\tau)d\tau \right]^2 = O\left(\frac{1}{N^2}\right).
            \end{aligned}
        \end{equation*}
        From (\ref{hatxi}), (\ref{limit FBSDE}) and (\ref{barx}), we get
        \begin{equation*}\left\{
            \begin{aligned}
                d(\hat{x}^{(N)} - \bar{x}) =& \left[ (A + G - B\varUpsilon^\dagger(P) \varPhi(P))(\hat{x}^{(N)} - \bar{x}) \right]dt\\
                & + \frac{1}{N} \sum_{j=1}^{N} \left[ C\hat{x}_j + D\hat{u}_j + F\hat{x}^{(N)} + \sigma \right]dW_j \\
                & + \left[ (C_0 + F_0 - D_0\varUpsilon^\dagger(P) \varPhi(P))(\hat{x}^{(N)} - \bar{x}) \right]dW_0 ,\\
                 \hat{x}^{(N)}(0) - \bar{x}(0) =&\ \frac{1}{N} \sum_{j=1}^{N} \xi_j - \bar{\xi}.
            \end{aligned}\right.
        \end{equation*}
        Thus
        \begin{equation*}
            \begin{aligned}
                \hat{x}^{(N)}(t) - \bar{x}(t) =&\ e^{(A+G- B\varUpsilon^\dagger(P) \varPhi(P))t}\Bigg(\frac{1}{N} \sum_{j=1}^{N} \xi_j - \bar{\xi}\Bigg) + \int_0^t e^{(A+G- B\varUpsilon^\dagger(P) \varPhi(P))(t-\tau)}\\
                & \times\Bigg\{ \frac{1}{N}\sum_{j=1}^{N}\big[\cdots\big]dW_j + \big[(\cdots)(\hat{x}^{(N)} - \bar{x})\big]dW_0 \Bigg\},\\
                \mathbb{E} \left[ \hat{x}^{(N)}(t) - \bar{x}(t) \right]^2 \le&\ \frac{c}{N} + c \mathbb{E} \left[ \int_0^T \left(\hat{x}^{(N)} - \bar{x}\right)^2 dt \right].
            \end{aligned}
        \end{equation*}
        By Gronwall's inequality, we have
        \begin{equation*}
            \mathbb{E} \left[ \hat{x}^{(N)}(t) - \bar{x}(t) \right]^2 \le \frac{c}{N}.
        \end{equation*}
        Thus
        \begin{equation*}
            \sup_{0 \le t \le T} \mathbb{E} \left[ \hat{x}^{(N)}(t) - \bar{x}(t) \right]^2 = O\left(\frac{1}{N}\right),
        \end{equation*}
        and then
        \begin{equation*}
            \sup_{0 \le t \le T} \mathbb{E} \left[ \hat{x}_i(t) - \bar{x}_i(t) \right]^2 = O\left(\frac{1}{N}\right).
        \end{equation*}
        Hence $\mathcal{I}_i^N = O\left(\frac{1}{\sqrt{N}}\right)$.
        The proof is complete.
    \end{proof}

    \subsection{The infinite-horizon problem}

    For the infinite-horizon case, we will only consider the time-invariant coefficients $A$, $B$, $G$, $C$, $D$, $F$, $C_0$, $D_0$, $F_0$, $Q$, $R$, $\varGamma$. That is the state of the $i$th agent is
    \begin{equation}\label{infinite state}
        \left\{
        \begin{aligned}
            dx_i(t) =& \big[Ax_i(t) + Bu_i(t) + Gx^{(N)}(t) + f(t)\big]dt\\
                     & + \big[Cx_i(t) + Du_i(t) + Fx^{(N)}(t) + \sigma(t)\big]dW_i(t)\\
                     & + \big[C_0x_i(t) + D_0u_i(t) + F_0x^{(N)}(t) + \sigma_0(t)\big]dW_0(t),\quad t\geq0,\\
            x_i(0) =&\ \xi_i ,
        \end{aligned}
        \right.
    \end{equation}
    and his/her cost functional is given by
    \begin{equation}\label{infinite cost}
        \begin{aligned}
            \bar{J}^N_i(u_i(\cdot), u_{-i}(\cdot)) = \mathbb{E} \int_0^\infty \left[\big|\big|x_i(t) - \varGamma x^{(N)}(t) - \eta(t)\big|\big|^2_{Q} +\big|\big|u_i(t)\big|\big|^2_{R} \right]dt,
        \end{aligned}
    \end{equation}
    where $Q \ge 0$ and $R > 0$, $f(\cdot),\sigma(\cdot),\sigma_0(\cdot) \in L^2_{\mathcal{F}^0}(0,\infty;\mathbb{R}^n)$.

    We study the following infinite-horizon problems.

    \noindent {\bf (PD')}: Find an $\epsilon$-Nash equilibrium strategy $u^{*N}(\cdot) := (u_1^*(\cdot),\cdots,u_N^*(\cdot))$ among $u^N(\cdot) \in \mathscr{U}_{d}[0,\infty)$ for (\ref{infinite cost}), subject to (\ref{infinite state}).

    Motivated by the discussion on finite-horizon problems in the previous section, we design the decentralized strategies
    \begin{equation}\label{hatui'}
        \hat{u}_i = -\varUpsilon^{-1}(P) \varPhi(P) \left( \hat{x}_i - \bar{x} \right) - \bar{\varUpsilon}^{-1}(P,K) \big[ (\varPhi(P) + \varPsi(P,K)) \bar{x} + \varTheta(P,K,\varphi,\psi) \big],
    \end{equation}
    $i = 1,\cdots,N$, where
    \begin{equation}\label{ARE}
        \left\{
        \begin{aligned}
            &PA + A^\top P + C^\top PC + C_0^\top PC_0 - \varPhi^\top(P) \varUpsilon^{-1}(P) \varPhi(P) + Q = 0,\\
            &K(A+G) + PG + A^\top K + \varPhi^\top(P) \varUpsilon^{-1}(P) \varPhi(P) - \bar{\varPhi}^\top(P,K) \bar{\varUpsilon}^{-1}(P,K) (\varPhi(P) + \varPsi(P,K))\\
            & + C^\top PF + C_0^\top PF_0 + C_0^\top K(C_0+F_0) - Q \varGamma = 0,\\
            &d\varphi = - \left[\left( A^\top - \bar{\varPhi}^\top(P,K) \bar{\varUpsilon}^{-1}(P,K) B^\top \right) \varphi + \left( C^\top - \bar{\varPhi}^\top(P,K) \bar{\varUpsilon}^{-1}(P,K) D^\top \right) P\sigma \right. \\
            &\qquad + \left. \left( C_0^\top - \bar{\varPhi}^\top(P,K) \bar{\varUpsilon}^{-1}(P,K) D_0^\top \right)(\psi + (P+K) \sigma_0) + (P+K) f - Q\eta \right] dt + \psi dW_0,\\
            &d\bar{x} = \left[\left( A + G - B \bar{\varUpsilon}^{-1}(P,K) (\varPhi(P) + \varPsi(P,K)) \right) \bar{x} - B \bar{\varUpsilon}^{-1}(P,K) \varTheta(P,K,\varphi,\psi) + f\right]dt \\
            &\qquad + \big[\left( C_0 + F_0 - D_0 \bar{\varUpsilon}^{-1}(P,K) (\varPhi(P) + \varPsi(P,K)) \right) \bar{x}\\
            &\qquad\quad - D_0 \bar{\varUpsilon}^{-1}(P,K) \varTheta(P,K,\varphi,\psi) + \sigma_0\big]dW_0,\qquad \bar{x}(0) = \bar{\xi},
        \end{aligned}
        \right.
    \end{equation}

    The following assumption is indispensable for the subsequent analysis.

    \noindent {\bf (A6)} The equation (\ref{ARE}) admits a set of stabilizing solution $(P(\cdot), K(\cdot), \varphi(\cdot), \psi(\cdot),\bar{x}(\cdot))$.

    We further analyze the sufficient conditions for assumption {\bf (A6)} to hold, for which we first introduce some definitions from \cite{Zhang-Chen-04}, \cite{Chen-Zhang-04}, \cite{Zhang-Zhang-Chen-08}.
    Consider the linear controlled stochastic system
    \begin{align}
            dx(t) =&\ \big[Ax(t)+Bu(t)\big]dt + \big[Cx(t)+Du(t)\big]dW(t) \nonumber \\
                   &\ + \big[C_0x(t)+D_0u(t)\big]dW_0(t),\quad t\geq0,\quad x(0)=x_0, \label{state'} \\
            y(t) =&\ \varXi x(t),\quad t\geq0,\label{observe}
    \end{align}
    where $y(\cdot) \in \mathbb{R}^n$ is the observation process with $\varXi\in\mathbb{R}^{n\times n}$, $x_0 \in \mathbb{R}^n$ is the initial state.

    \begin{mydef}
        The system (\ref{state'}) with $u(\cdot)\equiv0$ (or denoted by $[A,C,C_0]$) is said to be mean-square stable if for any $x_0$, there exists $c>0$ such that $\mathbb{E}\int_0^\infty||x(t)||^2dt \le c$.
    \end{mydef}

    \begin{mydef}
        The system (\ref{state'}) (or denoted by $[A,C,C_0;B,D,D_0]$) is said to be stabilizable if there exists a state feedback control $u(\cdot) = \varLambda x(\cdot)$, $\varLambda\in\mathbb{R}^{n\times n}$, such that for any $x_0$, the closed-loop system
        \begin{equation*}
            \begin{aligned}
                dx(t) &= (A+B\varLambda)x(t)dt + (C+D\varLambda)x(t)dW(t)\\
                      &\quad + (C_0+D_0\varLambda)x(t) dW_0(t),\quad t\geq0,\quad x(0)=x_0
            \end{aligned}
        \end{equation*}
        is mean-square stable.
    \end{mydef}

    \begin{mydef}
        The system (\ref{state'})-(\ref{observe}) with $u(\cdot)\equiv0$ (or denoted by $[A,C,C_0;\varXi]$) is said to be exactly detectable if $y(t) = \varXi x(t) \equiv 0$, a.s. $t \in [0,T],\forall T>0$, implies $\lim_{t \to \infty}\mathbb{E}[x(t)]^2=0$.
    \end{mydef}

    \begin{mydef}
        The system (\ref{state'})-(\ref{observe}) with $u(\cdot)\equiv0$ is said to be exactly observable if $y(t) = \varXi x(t) \equiv 0$, a.s. $t \in [0,T],\forall T>0$, implies $x_0=0$.
    \end{mydef}

    The following proposition gives some conditions to ensure that {\bf (A6)} holds.
    \begin{mypro}
        Assume the system $[A,C,C_0;B,D,D_0]$ is stabilizable, $[A,C,C_0;\sqrt{Q}]$ is exactly detectable (exactly observable), and the second equation of $K(\cdot)$ in (\ref{ARE}) admits a stabilizing solution, then {\bf (A6)} holds.
    \end{mypro}
    \begin{proof}
        Note that the first equation in (\ref{ARE}) is a symmetric {\it algebraic Riccati equation} (ARE).
        If the system $[A,C,C_0;B,D,D_0]$ is stabilizable, and $[A,C,C_0;\sqrt{Q}]$ is exactly detectable (exactly observable), then there exists a unique stabilizing solution $P(\cdot) \ge 0$ ($P(\cdot) > 0$) to the first equation in (\ref{ARE}) (\cite{AitRami-Zhou-00}).
        If the second equation in (\ref{ARE}) admits a stabilizing solution $\varPi(\cdot)$, then by \cite{Sun-Yong-18}, the third equation in (\ref{ARE}) has a stabilizing solution $(\varphi(\cdot),\psi(\cdot))$.
        Thus the fourth equation in (\ref{ARE}) exists a stabilizing solution $\bar{x}(\cdot)$.
    \end{proof}

    The second equation of $K(\cdot)$ in (\ref{ARE}) is a non-symmetric ARE, whose solvability is challenging.
    Interested readers may refer to \cite{Juang-Lin-99}, \cite{Guo-02}, \cite{Engwarda-05} and \cite{Ma-Lu-16}.
    Our analysis is provided below.
    Denote $\varPi(\cdot) := P(\cdot) + K(\cdot)$.
    Then from (\ref{ARE}), $\varPi(\cdot)$ satisfies
    \begin{equation}\label{Pi}
        \begin{aligned}
            &\varPi(A+G) + A^\top \varPi + C^\top P(C+F) + C_0^\top \varPi (C_0 + F_0) - \left(\varPi B + C^\top PD + C_0^\top \varPi D_0\right)\\
            \times &\left(R + D^\top PD + D_0^\top \varPi D_0\right)^{-1} \left(B^\top \varPi + D^\top P(C+F) + D_0^\top \varPi (C_0 + F_0)\right) + Q(I-\varGamma) = 0.
        \end{aligned}
    \end{equation}
    If the first equation in (\ref{ARE}) admits a stabilizing solution $P(\cdot)$ and (\ref{Pi}) admits a stabilizing solution $\varPi(\cdot)$, then the second equation in (\ref{ARE}) admits a stabilizing solution $K(\cdot) = \varPi(\cdot) - P(\cdot)$.
    So we only need to analyze the solvability of (\ref{Pi}) under the assumption that the first equation in (\ref{ARE}) is solvable.

    \begin{mypro}
        Assume that (\ref{Pi}) is symmetric, i.e., $G = lI$ with scalar $l$, $F = 0$, $F_0 = 0$ (no mean field term in the diffusion terms) and $\varGamma = \gamma I$ with scalar $\gamma$ such that $1-\gamma \ge 0$.
        If the system $[A,C_0;B,D_0]$ is stabilizable, $\big[A,C_0;\sqrt{Q(I-\varGamma)}\big]$ is exactly detectable (exactly observable), then (\ref{Pi}) admits a stabilizing solution $\varPi(\cdot)$.
    \end{mypro}
    \begin{proof}
        Under the assumptions, (\ref{Pi}) turns into
        \begin{equation*}
            \begin{aligned}
                &\varPi\left(A+\frac{l}{2}\right) + \left(A+\frac{l}{2}I\right)^\top \varPi + C^\top PC + C_0^\top \varPi C_0 - \left(\varPi B + C^\top PD + C_0^\top \varPi D_0\right)\\
                \times &\left(R + D^\top PD + D_0^\top \varPi D_0\right)^{-1} \left(B^\top \varPi + D^\top PC + D_0^\top \varPi C_0\right) + Q(I-\gamma I) = 0,
            \end{aligned}
        \end{equation*}
        we can see that it is symmetric.
        If the system $[A,C_0;B,D_0]$ is stabilizable, and $\big[A,C_0;\sqrt{Q(I-\varGamma)}\big]$ is exactly detectable (exactly observable), then there exists a unique stabilizing solution $\varPi(\cdot) \ge 0$ ($\varPi(\cdot) > 0$) to (\ref{Pi}) (\cite{AitRami-Zhou-00}).
    \end{proof}

    \begin{mypro}\label{c-splitting}
        {\bf(a)} Assume that $C_0 = kI$ with scalar $k$ and $D_0 = 0$, then (\ref{Pi}) admits a stabilizing solution $\varPi(\cdot)$ if and only if
        \begin{equation*}
            \mathcal{M}_a := \begin{bmatrix}
                A+G-B\tilde{\varUpsilon}^{-1}D^\top P(C+F)+k(C_0+F_0) & -B\tilde{\varUpsilon}^{-1}B^\top \\ C^\top PD\tilde{\varUpsilon}^{-1}D^\top P(C+F) - C^\top P(C+F) - Q(I-\varGamma) & -A^\top + C^\top PD\tilde{\varUpsilon}^{-1}B^\top
            \end{bmatrix},
        \end{equation*}
        where $\tilde{\varUpsilon} := R+D^\top PD$, is $(n,n)\ c$-splitting, i.e., both the open left half-plane and the open right half-plane contain $n$ eigenvalues.

        \noindent{\bf(b)} Assume that $C_0+F_0 = jI$ with scalar $j$ and $D_0 = 0$, then (\ref{Pi}) admits a stabilizing solution $\varPi(\cdot)$ if and only if
        \begin{equation*}
            \mathcal{M}_b := \begin{bmatrix}
                A+G-B\tilde{\varUpsilon}^{-1}D^\top P(C+F) & -B\tilde{\varUpsilon}^{-1}B^\top \\ C^\top PD\tilde{\varUpsilon}^{-1}D^\top P(C+F) - C^\top P(C+F) - Q(I-\varGamma) & -A^\top + C^\top PD\tilde{\varUpsilon}^{-1}B^\top - jC_0
            \end{bmatrix},
        \end{equation*}
        is $(n,n)\ c$-splitting.
    \end{mypro}
    \begin{proof}
        We only prove the first part of the proposition, as the second part can be proved similarly.
        The proof closely mirrors that of Theorem 18 in \cite{Huang-Zhou-20}.
        Let
        \begin{equation*}
        \begin{aligned}
            &\mathcal{L} := \begin{bmatrix}
                I & 0 \\ \varPi & I
            \end{bmatrix},\quad 
            \hat{A} := A+G-B\tilde{\varUpsilon}^{-1}\left(B^\top\varPi+D^\top P(C+F)\right)+k(C_0+F_0),\\
            &\check{A} := A^\top - \left(\varPi B+C^\top PD\right)\tilde{\varUpsilon}^{-1}B^\top.
        \end{aligned}
        \end{equation*}
        It can be verified that
        \begin{equation*}
            \mathcal{L}^{-1}\varPi\mathcal{L} := \begin{bmatrix}
                \hat{A} & -B\tilde{\varUpsilon}^{-1}B^\top \\ 0 & -\check{A}
            \end{bmatrix}.
        \end{equation*}
        If (\ref{Pi}) admits a stabilizing solution $\varPi(\cdot)$, then
        $\hat{A}$ and $\check{A}$ are both Hurwitz, meaning all eigenvalues of $\hat{A}$ and $\check{A}$ have negative real parts.
        Consequently, the matrix $\mathcal{M}_a$ is $(n,n)\ c$-splitting.
        Conversely, if $\mathcal{M}_a$ is $(n,n)\ c$-splitting, then $\hat{A}$ and $\check{A}$ both have $n$ eigenvalues in the open left half-plane, which implies that $\hat{A}$ and $\check{A}$ are both Hurwitz.
        Hence, (\ref{Pi}) admits a stabilizing solution $\varPi(\cdot)$.
    \end{proof}

    Similar to the finite-horizon case, we have the following result.
    \begin{mythm}\label{NEinfinite}
        Assume that (A1), (A2) and (A6) hold. For problem {\bf (PD')}, $( \hat{u}_1(\cdot), \cdots, \hat{u}_N(\cdot))$ given in (\ref{hatui'}) constitutes an $\epsilon$-Nash equilibrium, where $\epsilon = O\left(\frac{1}{\sqrt{N}}\right)$.
    \end{mythm}
    The proof is similar to that of Theorem \ref{NE}, hence it is omitted.

    \section{Numerical examples}

    In this section, a numerical simulation of Problem {\bf (PD)} is given to verify the conclusions we obtained in the previous section.
    The simulation parameters are listed as follows:
    \begin{align*}
        A& =
        \begin{bmatrix}
            0 & 1 \\
            -0.05 & -0.1
        \end{bmatrix},\quad
        B =
        \begin{bmatrix}
            2.5 \\
            1.6
        \end{bmatrix},\quad
        G =
        \begin{bmatrix}
            0.3 & 0.1 \\
            0.1 & 0.3
        \end{bmatrix},\quad
        C = C_0 =
        \begin{bmatrix}
            0 & -0.8 \\
            0.8 & 0.4
        \end{bmatrix},\\
        D& = D_0 =
        \begin{bmatrix}
            0.15 \\
            0.05
        \end{bmatrix},\quad
        F = F_0
        \begin{bmatrix}
            0.4 & 0.2 \\
            0.2 & 0.4
        \end{bmatrix},\quad
        \Gamma =
        \begin{bmatrix}
            0.5 & 0 \\
            0 & 0.5
        \end{bmatrix},\quad
        \Gamma_0 =
        \begin{bmatrix}
            0.2 & 0 \\
            0 & 0.2
        \end{bmatrix},\\
        Q& =
        \begin{bmatrix}
            2 & 1 \\
            1 & 2
        \end{bmatrix},\quad
        R = 0.3,\quad
        H =
        \begin{bmatrix}
            0.5 & 0 \\
            0 & 0.5
        \end{bmatrix},\quad
        f =
        \begin{bmatrix}
            0.05t \\
            0.05t^{\frac{3}{2}}
        \end{bmatrix},\\
        \sigma& =
        \begin{bmatrix}
            0.05\sin t \\
            0.05t^{\frac{1}{2}}
        \end{bmatrix},\quad
        \sigma_0 =
        \begin{bmatrix}
            0.05t \\
            0.05t^{\frac{1}{2}}
        \end{bmatrix},\quad
        \eta =
        \begin{bmatrix}
            0.1 t \\
            0.1t^{\frac{1}{2}}
        \end{bmatrix},\quad
        \eta_0 =
        \begin{bmatrix}
            0.5 \\
            0.5
        \end{bmatrix}
    \end{align*}
    Consider a two-dimensional system with 50 agents.
    The time interval is set to $[0,10]$.
    Each agent's initial state is drawn independently and identically: the first component is uniformly distributed on $[0,5]$, and the second on $[0,2]$.
    The curves of Riccati equations $P(\cdot)$ and $K(\cdot)$, described by (\ref{P}), and (\ref{K}), are shown in Figure \ref{PK_TwoDimension}.
    \begin{figure}[H]
        \centering\includegraphics[width=10cm]{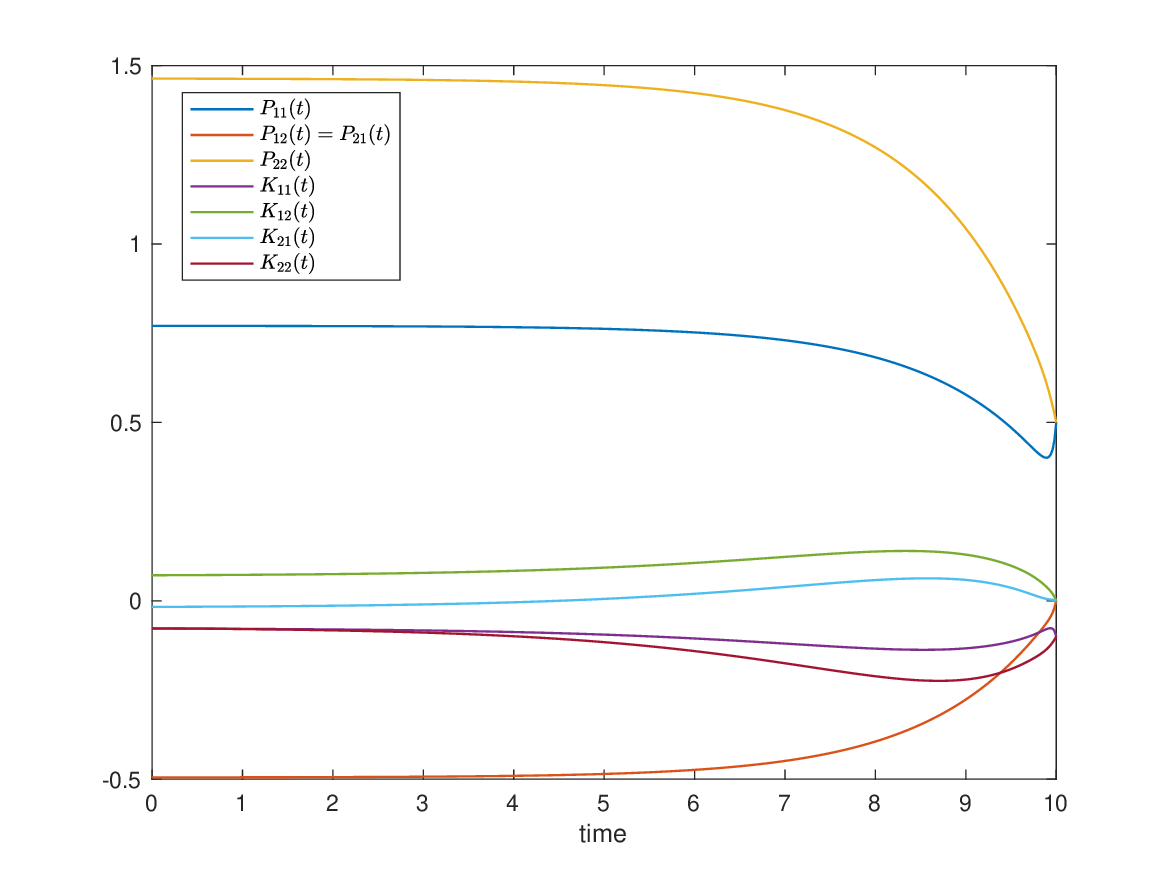}
        \caption{The curves of $P(\cdot)$ and $K(\cdot)$}
        \label{PK_TwoDimension}
    \end{figure}

    We denote the performance of the decentralized strategy by $\epsilon(N) = \left(\mathbb{E}\int_{0}^{T}||\hat{x}^{(N)} - \bar{x}||^2 dt\right)^{\frac{1}{2}}$.
    The graph of $\epsilon(N)$ versus $N$ is shown in Figure \ref{epsilon_TwoDimension}.
    It can be seen that as $N$ increases, $\epsilon(N)$ decreases significantly, which confirms the consistency of the mean-field approximation.
    \begin{figure}[H]
        \centering\includegraphics[width=10cm]{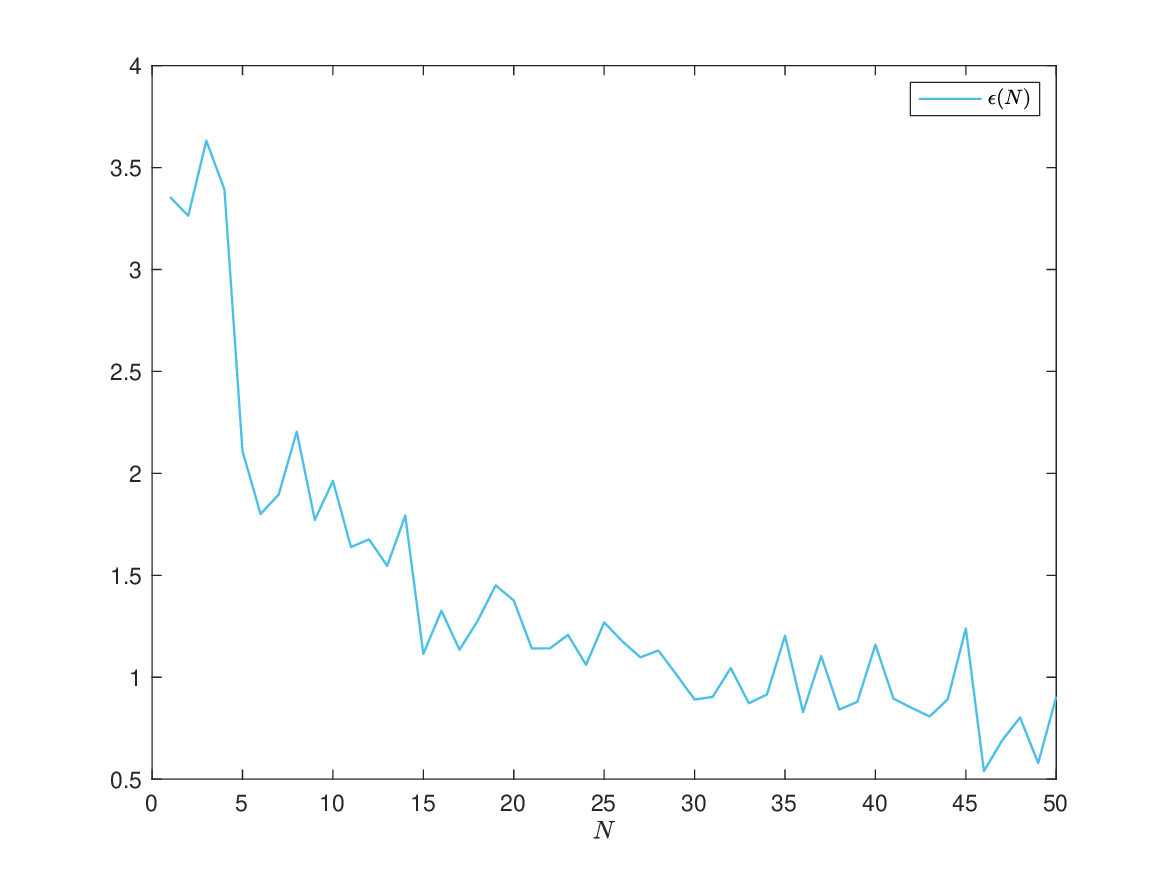}
        \caption{$\epsilon(N)$ with respect to $N$}
        \label{epsilon_TwoDimension}
    \end{figure}

    Under strategy (\ref{hatui}), the trajectories of the 50 agents are shown separately for each state dimension: Figure \ref{50trajectory_FirstDimension} displays the first component, and Figure \ref{50trajectory_SecondDimension} the second, together with the average state $\hat{x}^{(N)}(\cdot)$ and mean-field approximation $\bar{x}(\cdot)$.
    \begin{figure}[H]
        \centering\includegraphics[width=10cm]{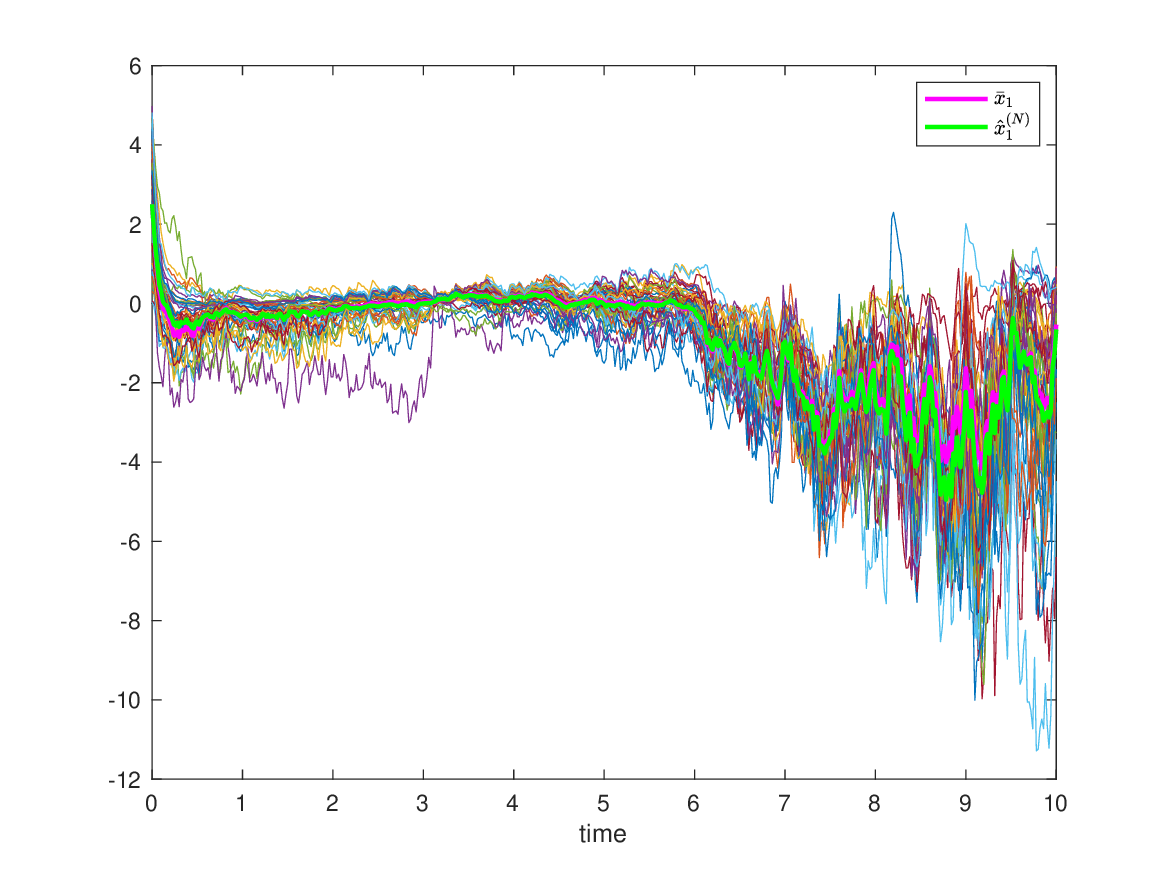}
        \caption{Trajectories of the 50 agents' first components, together with their average $\hat x^{(N)}(\cdot)$ and the mean-field approximation $\bar x(\cdot)$}
        \label{50trajectory_FirstDimension}
    \end{figure}
    \begin{figure}[H]
        \centering\includegraphics[width=10cm]{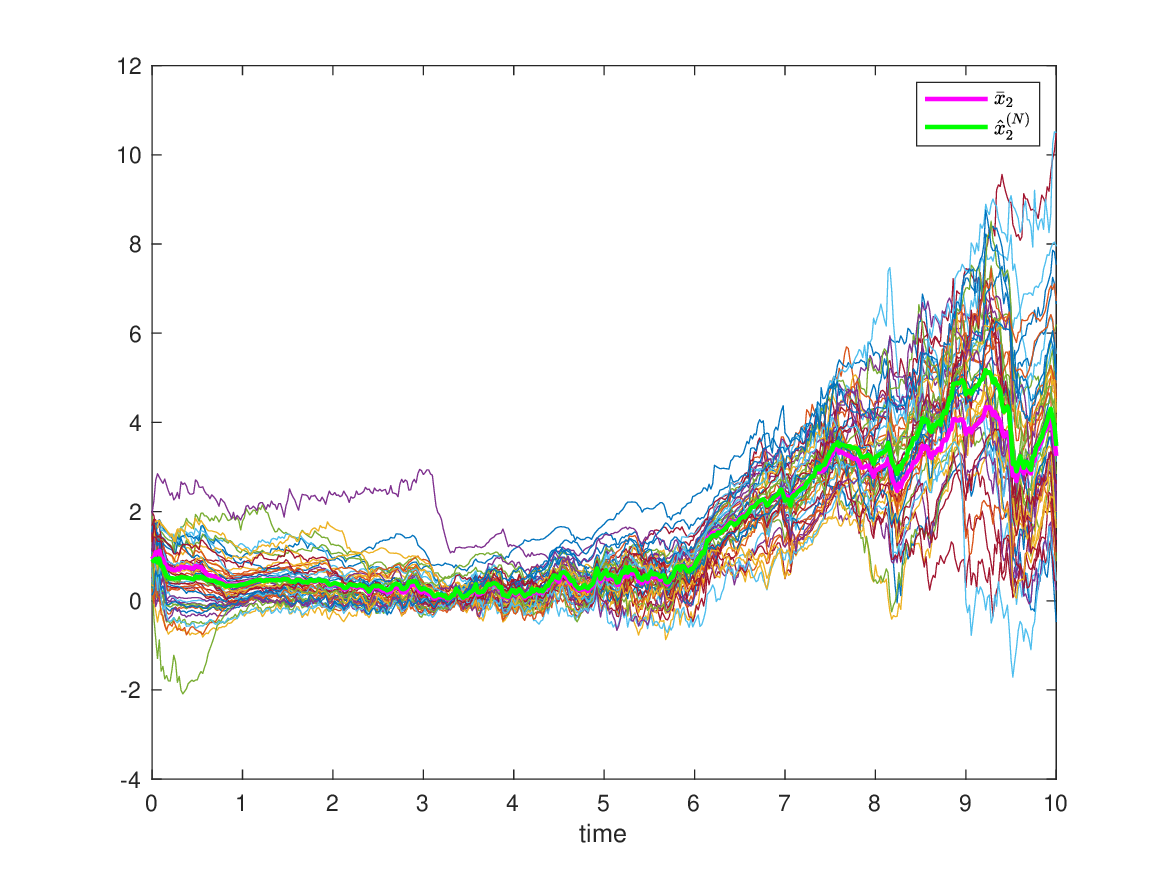}
        \caption{Trajectories of the 50 agents' second components, together with their average $\hat x^{(N)}(\cdot)$ and the mean-field approximation $\bar x(\cdot)$}
        \label{50trajectory_SecondDimension}
    \end{figure}
    We can observe that the trajectories of the 50 agents reveal similar trends, and their average state remains very close to the mean-field approximation.

    \section{Conclusions}

    In this paper, we explore a more general LQ-MFGs problem complicated by multiplicative and common noises.
    Using the direct approach, we effectively address computational complexity issues arising from the increase in adjoint equations due to mean-field coupling.
    We design decentralized control laws for both finite-horizon and infinite-horizon scenarios, proving them to be $\epsilon$-Nash equilibria.
    In future work, we will explore LQ Stackelberg MFGs problem with multiplicative and common noises, and analyze their practical applications.

\end{document}